\RequirePackage{fix-cm}
\documentclass[smallextended]{svjour3}       %
\smartqed  %
\usepackage{graphicx}
\usepackage{fancybox,ulem,subfig}
\usepackage{yuyuan}
\usepackage{booktabs}
\usepackage{hyperref}
\usepackage{algorithm, algpseudocode, multirow}
\DeclareMathAlphabet{\mathpzc}{OT1}{pzc}{m}{it}
\newcommand\qedsymbol{\begin{flushright}$\blacksquare$\end{flushright}}
\usepackage[a4paper, margin=1.2in]{geometry}

\numberwithin{lemma}{section}
\numberwithin{corollary}{section}
\numberwithin{theorem}{section}
\numberwithin{proposition}{section}
\numberwithin{algorithm}{section}
\numberwithin{assumption}{section}
\numberwithin{equation}{section}
\numberwithin{remark}{section}

\begin{document}
\global\long\def\xt#1{x^{#1}}%
\global\long\def\yt#1{y^{#1}}%
\global\long\def\zt#1{z^{#1}}%
\global\long\def\norm#1{\|#1\|}%
\global\long\def\normsq#1{\|#1\|^{2}}%
\global\long\def\N{\mathbb{N}}%
\global\long\def\normcb#1{\|#1\|^{3}}%
\global\long\def\inner#1#2{\langle#1,#2\rangle}%
\global\long\def\grad{\nabla}%
\global\long\def\Dely{\Delta_{y}}%
\global\long\def\oneroottwo{\frac{1}{\sqrt{2}}}%
\global\long\def\lg#1#2{l_{g}(y^{#1};y^{#2})}%
\global\long\def\lf#1#2{l_{f}(x^{#1};x^{#2})}%
\global\long\def\lh#1#2{l_{h}(z^{#1};z^{#2})}%
\global\long\def\argmin{\arg\min}%
\global\long\def\ep{\epsilon}%
\global\long\def\R{\mathbb{R}}%
\global\long\def\ceil#1{\Big\lceil#1\Big\rceil}%
\global\long\def\R{\mathbb{R}}%
\global\long\def\ytil{\tilde{y}}%
\global\long\def\Conv{\text{Conv}}%
\global\long\def\half{\frac{1}{2}}%
\global\long\def\lam{\lambda}%
\global\long\def\xunder{\underline{x}}%
\global\long\def\alpt#1{\alpha_{#1}}%
\global\long\def\etat#1{\eta_{#1}}%
\global\long\def\xbar{\bar{x}}%
\global\long\def\gradsq{\grad^{2}}%

\global\long\def\Ltil{\tilde{L}}%
\global\long\def\xstar{x^{*}}%

\global\long\def\fstar{f^{*}}%
\global\long\def\wt#1{\omega_{#1}}%
\global\long\def\dist{\text{dist}}%
\global\long\def\Xstar{X^{*}}%

\global\long\def\xhat{\hat{x}}%
\global\long\def\xstar{x^{*}}%
\global\long\def\Ne{\mathcal{N}}%
\global\long\def\Lmax{L_{\max}}%
\global\long\def\mumin{\mu_{\min}}%
\global\long\def\normsq#1{\|#1\|^{2}}%
\global\long\def\funder{\underline{f}}%
\global\long\def\ftil{\tilde{f}}%
\global\long\def\levelf#1{\mathcal{E}_{f}(#1)}%
\global\long\def\Xtil{\tilde{X}}%
\global\long\def\xbrt#1{x^{(#1)}}%
\global\long\def\tiprime#1{t_{#1}^{'}}%
\global\long\def\ti#1{t_{#1}}%
\global\long\def\xti#1{x^{t_{#1}}}%
\global\long\def\xtipr#1{x^{t_{#1}^{'}}}%
\global\long\def\tipr#1{t_{#1}^{'}}%
\global\long\def\lfj{l_{f_{j}}}%
\global\long\def\dist{\text{dist}}%
\global\long\def\distsq{\text{dist}^{2}}%
\global\long\def\normvsq#1{\|#1\|_{V}^{2}}%

\global\long\def\distsq#1{\text{dist}^{2}(#1)}%
\global\long\def\frN{\mathfrak{N}}%
\global\long\def\calM{\mathcal{M}}%
\global\long\def\yunderti#1#2{\underline{y}^{#1,[#2]}}%
\global\long\def\ybarti#1#2{\bar{y}^{#1,[#2]}}%
\global\long\def\yti#1#2{y^{#1,[#2]}}%
\global\long\def\ybar{\bar{y}}%
\global\long\def\lti#1#2{l^{#1,[#2]}}%

\global\long\def\alpt{\alpha_{t}}%
\global\long\def\wt#1{\omega_{#1}}%
\global\long\def\tildef{\ftil}%
\global\long\def\ibar{\bar{i}}%
\global\long\def\calK{\mathcal{K}}%
\global\long\def\calP{\mathcal{P}}%
\global\long\def\rint{\text{rint}}%
\global\long\def\xtil{\tilde{x}}%
\global\long\def\aff{\text{Affine}}%
\global\long\def\Do{D_{0}}%
\global\long\def\DX{D_{X}}%
\global\long\def\lj#1{f_{[#1]}}%

\global\long\def\li{l_{i}}%
\global\long\def\ri{r_{i}}%
\global\long\def\xtj#1#2{x_{#1}^{#2}}%
\global\long\def\xjt#1#2{x_{#1}^{#2}}%
\global\long\def\Mj{\calM_{j}}%
\global\long\def\jbar{\bar{j}}%
\global\long\def\fj#1{f_{[#1]}}%
\global\long\def\fjt#1#2{f_{#1}^{#2}}%
\global\long\def\Xj#1{X_{[#1]}}%

\global\long\def\lftil{\tilde{l}_{f}}%
\global\long\def\yt#1{y^{#1}}%
\global\long\def\tbar{\bar{t}}%
\global\long\def\quarter{\frac{1}{4}}%
\global\long\def\third{\frac{1}{3}}%
\global\long\def\qbar{\bar{q}}%
\global\long\def\Ltil{\tilde{L}}%
\global\long\def\Stil{\tilde{S}}%
\global\long\def\fbar{\bar{f}}%
\global\long\def\funder{\underline{f}}%
\global\long\def\ft#1{f^{#1}}%
\global\long\def\fundert#1{\underline{f}^{#1}}%
\global\long\def\fbart#1{\bar{f}^{#1}}%
\global\long\def\tauhat{\hat{\tau}}%
\global\long\def\taubar{\bar{\tau}}%
\global\long\def\ups{\upsilon}%
\global\long\def\Fups{F_{2\ups}}%
\global\long\def\fups{f_{2\ups}}%
\global\long\def\Punders#1{\underline{P}_{#1}}%
\global\long\def\calK{\mathcal{K}}%
\global\long\def\calS{\mathcal{S}}%
\global\long\def\sbar{\bar{s}}%
\global\long\def\xbart#1{\xbar^{#1}}%
\global\long\def\Otil{\tilde{O}}%
\global\long\def\calF{\mathcal{F}}%

\global\long\def\io{\iota}%
\global\long\def\Vio{\mathcal{V}_{\io}}%
\global\long\def\bigM{\mathbb{\mathbb{M}}}%

\global\long\def\Vdel{\Vio}%
\global\long\def\calV{\mathcal{V}}%
\global\long\def\lambar{\bar{\lambda}}%
\global\long\def\xplusdel{x_{\io}^{+}}%
\global\long\def\xplusio{x_{\io}^{+}}%
\global\long\def\xiop{x_{\io}^{+}}%
\global\long\def\mutil{\tilde{\mu}}%
\global\long\def\upstil{\tilde{\ups}}%
\global\long\def\mbar{\bar{m}}%
\global\long\def\sigbar{\bar{\sigma}}%

\global\long\def\iotamax{\iota_{\max}}%
\global\long\def\iomax{\iota_{\max}}%
\global\long\def\Frho{F_{2\rho}}%
\global\long\def\xplus{x^{+}}%
\global\long\def\tab{\quad}%
\global\long\def\Pbar{\bar{P}}%
\global\long\def\Punder{\underline{P}}%
\global\long\def\xstarp{x_{p}^{*}}%
\global\long\def\Delbar{\bar{\Delta}}%
\global\long\def\rhotil{\tilde{\rho}}%
\global\long\def\Lhat{\hat{L}}%
\global\long\def\rhomax{\rho_{\max}}%

\newcommand{\GR}{\mathcal{GR}}
\newcommand{\tGR}{$\mathcal{GR}$~}

\newcommand{\Lbar}{\bar{L}}

\makeatletter
\newcommand{\algmargin}{\the\ALG@thistlm}
\makeatother
\newlength{\whilewidth}
\settowidth{\whilewidth}{\algorithmicwhile\ }
\algdef{SE}[parWHILE]{parWhile}{EndparWhile}[1]
  {\parbox[t]{\dimexpr\linewidth-\algmargin}{%
     \hangindent\whilewidth\strut\algorithmicwhile\ #1\ \algorithmicdo\strut}}{\algorithmicend\ \algorithmicwhile}%
\algnewcommand{\parState}[1]{\State%
  \parbox[t]{\dimexpr\linewidth-\algmargin}{\strut #1\strut}}

 \renewcommand{\algorithmicrequire}{\textbf{Input:}}
\renewcommand{\algorithmicensure}{\textbf{Output:}}

\renewcommand{\emph}[1]{\textit{\textbf{#1}}}

\newcommand{\oldalglinenumber}{}
\NewDocumentCommand{\NoLNIf}{ m }{%
  \RenewCommandCopy{\oldalglinenumber}{\alglinenumber}
  \RenewDocumentCommand{\alglinenumber}{ m }{}
  \If{#1}
  \addtocounter{ALG@line}{-1}
  \RenewCommandCopy{\alglinenumber}{\oldalglinenumber}
}
\NewDocumentCommand{\NoLNEndIf}{}{%
  \RenewCommandCopy{\oldalglinenumber}{\alglinenumber}
  \RenewDocumentCommand{\alglinenumber}{ m }{}
  \EndIf
  \addtocounter{ALG@line}{-1}
  \RenewCommandCopy{\alglinenumber}{\oldalglinenumber}
}
\NewDocumentCommand{\NoLNElse}{}{%
  \RenewCommandCopy{\oldalglinenumber}{\alglinenumber}
  \RenewDocumentCommand{\alglinenumber}{ m }{}
  \Else
  \addtocounter{ALG@line}{-1}
  \RenewCommandCopy{\alglinenumber}{\oldalglinenumber}
}

\newcommand{\delbar}{\bar{\delta}}
\newcommand{\myl}{\mathpzc{l}}

\global\long\def\red#1{\textcolor{red}{#1}}%

\global\long\def\blue#1{\textcolor{blue}{#1}}%

\newcommand{\lihat}{\hat{l}_i}
\newcommand{\rihat}{\hat{r}_i}
\newcommand{\calW}{\mathcal{W}}
\newcommand{\tcW}{$\mathcal{W}$}
\title{Linearly Convergent Algorithms for Nonsmooth Problems with Unknown Smooth
    Regions}

\author{Zhe Zhang \and Suvrit Sra
}

\institute{
    Z. Zhang \at
    School of Industrial Engineering, Purdue University, USA \\
    \email{zhan5111@purdue.edu}
    \and
    S. Sra \at
    Department of Mathematics, CIT, TU Munich, Germany \\
    \email{s.sra@tum.de}
}

\date{Received: date / Accepted: date}

\maketitle

\begin{abstract}
    We develop efficient algorithms for optimizing piecewise smooth (PWS) functions when their smooth regions are \emph{unknown} to the algorithm. For convex PWS functions satisfying a quadratic growth (QG) condition, we demonstrate that the bundle-level (BL)-type method~\cite{lemarechal1995new} achieves global linear convergence---to our knowledge, the first such result for any algorithm for this problem class. We extend the method to solve weakly convex PWS problems, achieving the same order of complexity as in smooth non-convex optimization. Furthermore, we introduce a verifiable and accurate termination criterion for PWS optimization. Like the gradient norm in smooth optimization, this certificate tightly characterizes the optimality gap under the QG condition and can be evaluated without knowledge of any problem parameters. We develop a search subroutine for this certificate and embed it within a guess-and-check framework, resulting in almost parameter-free algorithms for both convex QG and weakly convex PWS problems.
\end{abstract}

\section{Introduction }

Non-smoothness is a major bottleneck in optimization. Without smoothness, convergence rates deteriorate from linear to sublinear for convex problems satisfying the quadratic growth (QG) condition~\cite{LanBook,nesterov2003introductory,nemirovsky1983problem}, while oracle complexity worsens by orders of magnitude for non-convex problems~\cite{zhang2020complexity,davis2019stochastic}. For general non-smooth functions, these losses are fundamentally unavoidable~\cite{nemirovsky1983problem,nesterov2003introductory}; achieving faster convergence therefore requires exploiting additional problem structure.

In this paper, we study a particular, widely applicable class of non-smooth functions: piecewise smooth (PWS) functions. A function is PWS if its domain can be covered by finitely many regions such that the function is smooth when restricted to each region (see Definition~\ref{def:piecewise-smooth}). This structure arises in many important applications, including statistics (e.g., elastic net regularization~\cite{zou2005regularization}), signal processing (e.g., compressive sensing~\cite{donoho2006compressed} and phase retrieval~\cite{duchi2019solving}), economics (e.g., matrix games~\cite{Nestrov2004Smooth}), control (e.g., multiparametric programming~\cite{borrelli2017predictive}), and machine learning (e.g., the ReLU activation function~\cite{montufar2014number}). In these applications, however, the underlying regions may be difficult to characterize or unavailable to the algorithm, motivating the following central research question:
\[\ovalbox{\begin{minipage}{0.8\columnwidth - 2\fboxsep - 0.8pt}%
            \centering \it Can one optimize PWS functions without knowing their underlying regions while achieving nearly the same oracle complexity as in smooth optimization?
        \end{minipage}}
\]

We answer this question by developing algorithms for PWS optimization problems of the form
\begin{equation}
    \min_{x\in X} f(x), \label{eq:opt_prob}
\end{equation}
where $f$ is PWS and $X$ is a simple\footnote{Here, ``simple'' means that Euclidean projection onto $X$ can be computed efficiently.}, closed and convex feasible region~\cite{ben2001lectures}. As an example, consider
\begin{equation}
    f(x) = \|x\|^2 + |x_1|, \quad x \in \mathbb{R}^2. \label{eq:demo-eg}
\end{equation}
Shown in Figure~\ref{fig:piecewise-smooth-eg}, the sign of $x_1$ partitions $\mathbb{R}^2$ into two regions: $X_1 := (-\infty,0]\times\mathbb{R}$ and $X_2 := (0,\infty)\times\mathbb{R}$. The function is differentiable everywhere except for the $x_2$-axis, and its gradient is continuous in the interior of each region.

\begin{figure}
    \subfloat[$f(x)=\protect\norm x^{2}+|x_{1}|.$\label{fig:piecewise-smooth-eg}]{\includegraphics[width=0.55\linewidth]{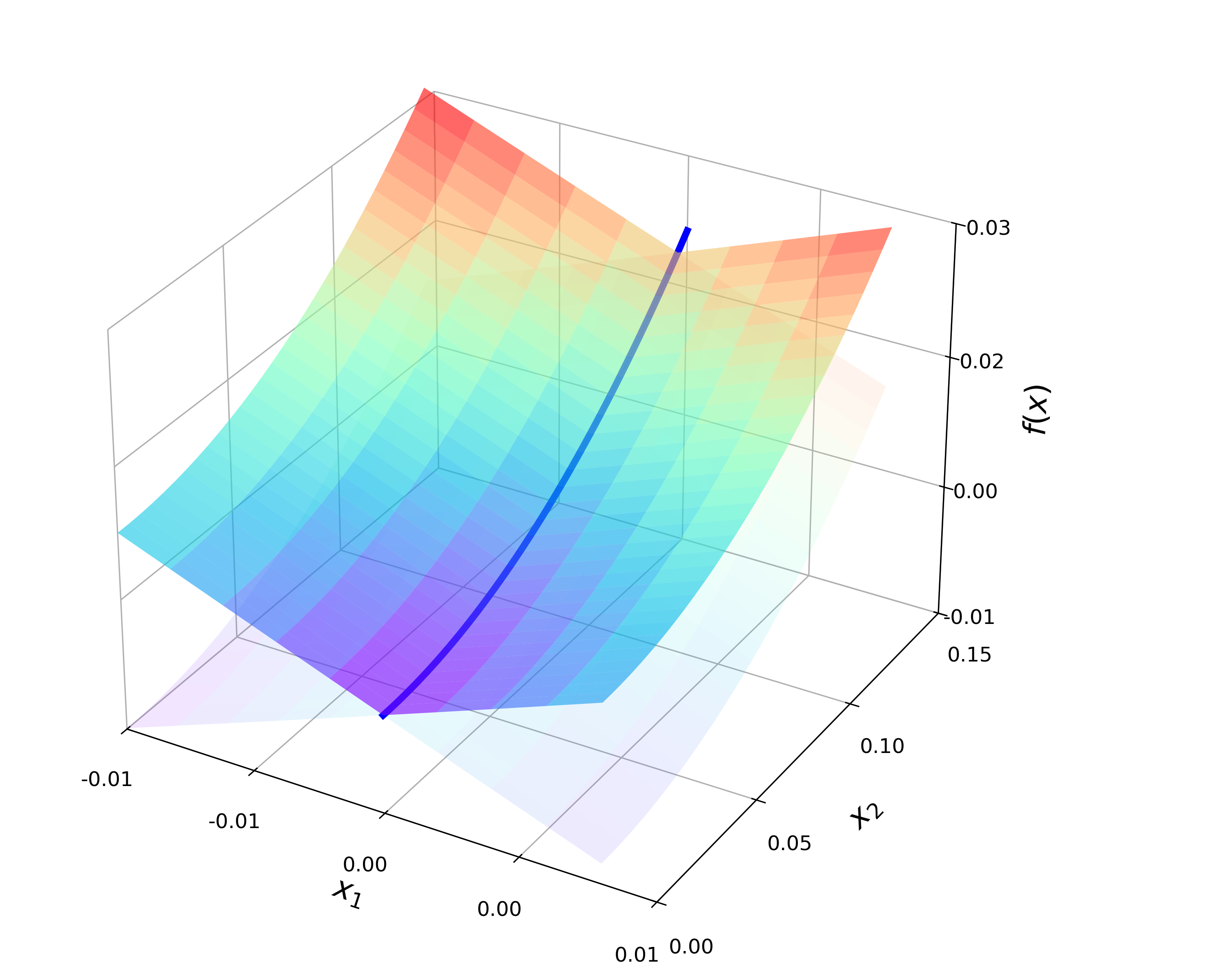}}\hspace*{\fill}\subfloat[{GD (Polyak Stepsize) from $x^{0}=[1e^{-4};1e^{-2}].$\label{fig:GD-Iterates-On-Piecewise-Smooth}}]{\includegraphics[width=0.55\linewidth]{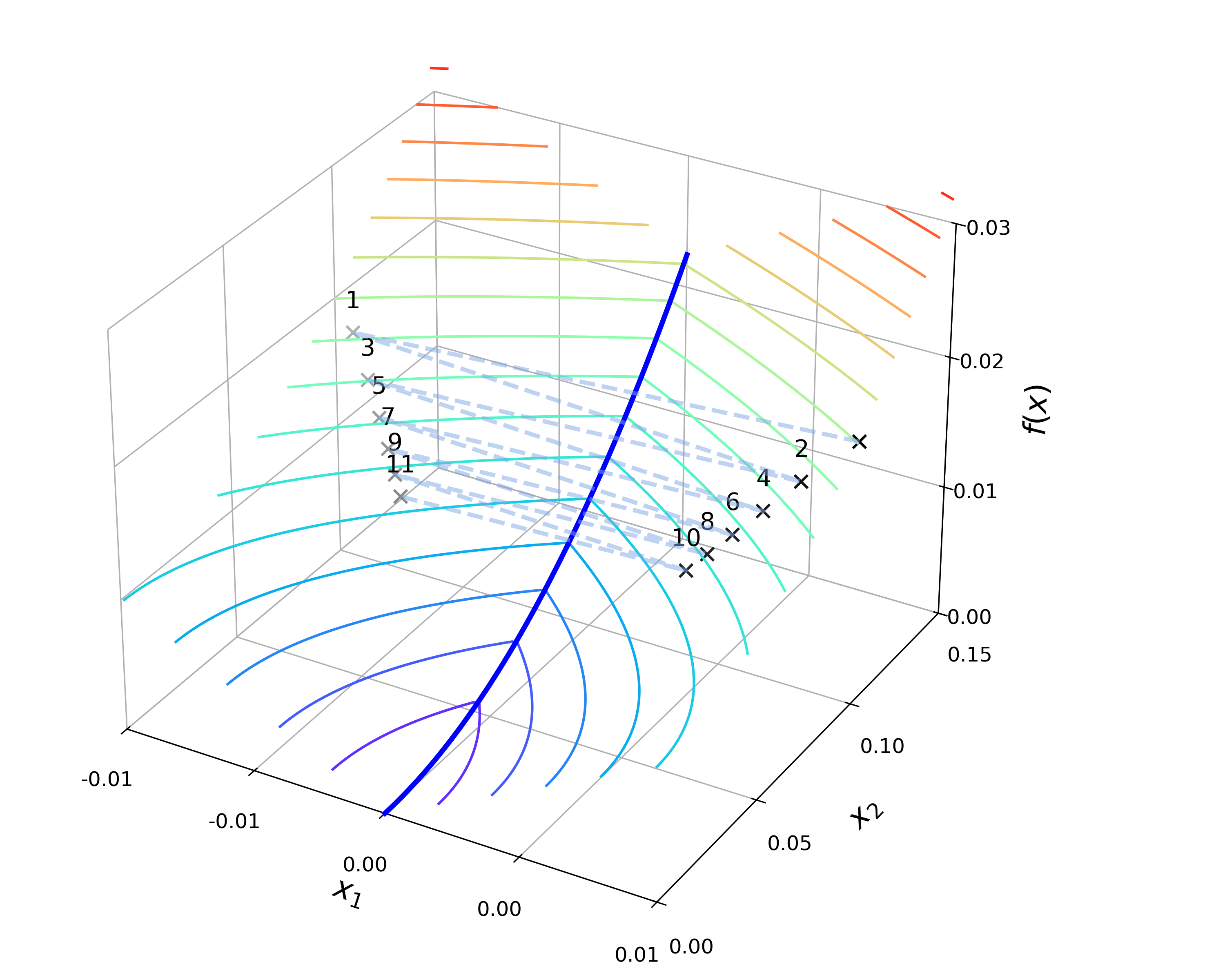}}\caption{Illustration of a Piecewise Smooth Function.}
\end{figure}

Despite the structure of PWS functions, there has been limited success in designing algorithms that achieve ``almost smooth performance'' when optimizing them, especially when the regions are unknown. Specifically, for iterative first-order (FO) methods, fast convergence typically requires computing the next iterate using accurate FO information. However, obtaining such information is challenging for PWS functions with unknown regions: the FO information (e.g., the gradient) generated at a given iterate is accurate only within the smooth region containing that iterate. When iterates move between different unknown smooth regions, it is difficult to ensure that this information remains relevant for computing the next iterate. More concretely, the standard gradient descent (GD) method can perform poorly even on the simple example in \eqref{eq:demo-eg}. Since the GD update relies solely on the gradient at the current iterate, the algorithm could stall when the next iterate lands in a different smooth region, leading to the characteristic zig-zagging illustrated in Figure~\ref{fig:GD-Iterates-On-Piecewise-Smooth}.

One natural way to overcome this limitation is to use information from \textit{multiple} past iterates. The bundle-level (BL) method combines FO information from the $m$ most recent iterates to construct level-set constraints (see \eqref{eq:proto-bl-update} below). When the next iterate $\xt{t+1}$ lands in the same region as a recent iterate $\xt{t-i}$ ($i \in \{0, \dots, m-1\}$), the constraint associated with $\xt{t-i}$ may help $\xt{t+1}$ make significant progress. In practice, the BL method has shown strong empirical performance and is a common choice in stochastic programming~\cite{shapiro2014lectures}. However, its superior performance over GD-type methods still lacks rigorous theoretical justification.

\subsection{Overview of Existing approaches}
Before detailing our approach and contributions, we first review existing approaches more broadly. In the literature, two main directions exist for tackling the PWS optimization problem \eqref{eq:opt_prob}, though neither approach fully achieves efficient global convergence when the function's underlying PWS structure is unknown. First, when the PWS function admits a known max-of-smooth representation (illustrated by the two intersecting graphs in Figure~\ref{fig:piecewise-smooth-eg}), the prox-linear update~\cite{nemirovski1995information,nesterov2003introductory,Lan15Bundle} can be employed. This method evaluates the gradients of all smooth component functions at the current iterate and uses this information to generate the next iterate. It can achieve ``almost smooth performance'' because at least one of these gradients provides accurate information for the active (maximal) component function at the next iterate. However, requiring access to all component functions is far too strong an assumption in applications where the regions are unknown.

The second approach exploits the so-called $\mathcal{VU}$ structure~\cite{mifflin2005algorithm,lemarechal2000,han2023survey,davis2024local}, which characterizes the local geometry around an optimal solution. Locally, the domain can be decomposed into two complementary subspaces, $\R^n = \mathcal{U} \oplus \mathcal{V}$, such that non-smoothness is restricted to translations of the $\mathcal{V}$ subspace. The $\mathcal{VU}$-type algorithm then jointly exploits sharp growth along the non-smooth $\mathcal{V}$ subspace and the smooth behavior of $f$ along the $\mathcal{U}$ subspace to achieve local linear convergence. These fast convergence rates apply locally and do not, by themselves, quantify the oracle cost of reaching the neighborhood where those rates hold.

\subsection{Main Contributions}

In this paper, we affirmatively answer the research question by re-examining the classical bundle-level (BL) method~\cite{lemarechal1995new} from a multi-step perspective.

The BL method generates the new iterate $\xt{t+1}$ by projecting the current iterate $\xt{t}$ onto a level set constructed from the preceding $m$ cuts (linearizations based on past gradients):
\begin{align}\label{eq:proto-bl-update}
    \begin{split}
        \xt{t+1} \leftarrow \argmin_{x\in X}\  & \frac{1}{2} \norm{x-\xt{t}}^{2}                                                                    \\ %
        \text{s.t.} \quad                      & f(\xt{t-i})+\inner{\grad f(\xt{t-i})}{x-\xt{t-i}} \leq l \quad \forall i \in \{0, 1, \dots, m-1\}. %
    \end{split}
\end{align}
Crucially, this update \eqref{eq:proto-bl-update} depends only on the algorithm parameters $m$ (number of cuts) and $l$ (the level value). It does not require explicit knowledge of the PWS function's underlying regions, making it directly applicable to our unknown PWS setting. %

Compared with earlier one-step analyses of bundle-level methods~\cite{lemarechal1995new,BenNem2005NERM,Lan15Bundle}, our analysis introduces a new tool for studying the progress made by the BL method across multiple iterations. Using this tool, we show that the BL method makes significant progress once two iterates (not necessarily consecutive) land within the same region. When the total number of regions is finite, say $k$, and the earlier cuts are retained, the pigeonhole principle guarantees such progress every $k+1$ steps, yielding the desired linear convergence under the QG condition.

Based on this method and our novel multi-step analysis, this paper contributes to the complexity theory of nonsmooth PWS optimization along the following directions:

\paragraph{\textsf{i).} A New Perspective and Global Linear Convergence.}
For a convex objective function $f$ satisfying the quadratic growth condition, we provide a simple proof that the single-loop BL method achieves global linear convergence when the optimal objective value $f^*$ is known and the number of cuts $m$ exceeds the total number of regions $k$. Our analysis further shows that the complexity bound depends on how often nearby iterates land within the same region, rather than solely on the global number of regions $k$.

\paragraph{\textsf{ii).} Generalization and Local Linear Convergence.}
To handle broader settings and relax assumptions on the oracle, we introduce a class of approximately PWS (apx-PWS) functions. In this setting, we show that the BL method retains the optimal convergence rate for non-smooth optimization. Additionally, we prove that once the iterates enter a neighborhood of the solution, the algorithm accelerates to a \textit{local linear convergence} rate if $m$ exceeds the local number of regions $k_{\text{local}}$. This provides theoretical justification for the superior empirical performance of bundle-level methods over methods based on gradient descent.

\paragraph{\textsf{iii).} Adaptive Methods and Application to Non-Convex Problems.}
To remove the impractical requirement of knowing the optimal objective value $f^*$, we propose an adaptive method, $\mu$-BL, that instead requires only the QG modulus $\mu$. Under this weaker requirement, $\mu$-BL achieves the same oracle complexity as the original BL method. We further extend this approach to non-smooth, non-convex problems by using $\mu$-BL as a subroutine within an inexact proximal point framework. The resulting method reduces the oracle complexity for finding an $\epsilon$-Moreau stationary point from $O(1/\epsilon^4)$ to $O(1/\epsilon^2)$, matching the optimal order for smooth non-convex optimization.

\paragraph{\textsf{iv).} A Verifiable Stationarity Certificate and Parameter-Free Algorithms.}
We introduce a novel \tcW-stationarity certificate that serves as a verifiable and accurate termination criterion for PWS optimization. Like the gradient norm in smooth optimization, this certificate tightly characterizes the optimality gap under the QG condition and can be computed without prior knowledge of problem parameters. The certificate also enables us to exploit unknown growth conditions. By embedding its search subroutine within a guess-and-check framework, we obtain almost parameter-free BL algorithms for both convex QG and weakly convex problems.

\paragraph{Organization of the Paper.}
The remainder of this section reviews related literature and introduces the notation and assumptions. Section~\ref{sec:fstar} presents the BL method and its convergence analysis when $f^*$ is known. Section~\ref{sec:inexact} extends the analysis to approximately PWS functions. Section~\ref{sec:QG} develops the adaptive method for a known QG modulus and applies it to weakly convex problems. Section~\ref{sec:termination} introduces the \tcW-stationarity certificate and a procedure for computing it. Finally, Section~\ref{sec:pf} uses this certificate to develop almost parameter-free algorithms. Auxiliary technical results are given in the appendix.

\subsection{Related Literature}

\paragraph{Bundle Type Algorithms.}

The cutting plane method, introduced in the 1960s~\cite{kelley1960cutting}, constructs a piecewise linear model of the objective function by taking the maximum of linear approximations (cuts) generated at previous iterates. It then selects a minimizer of this model as the next iterate. However, the iterates can change substantially as new cuts are added, leading to unstable performance in practice.

Lemaréchal, Nesterov, and Nemirovski~\cite{lemarechal1995new} introduced the bundle-level (BL) method to stabilize cutting plane updates by projecting the current iterate onto a level set of the cutting plane model. They established an $O(1/\epsilon^2)$ oracle complexity for general Lipschitz-continuous, non-smooth convex objectives. Subsequent extensions incorporate non-Euclidean Bregman distances~\cite{BenNem2005NERM} and achieve accelerated convergence rates~\cite{Lan15Bundle}. A more recent BL-type method for function-constrained optimization has shown promising numerical results~\cite{deng2024uniformly}. Nemirovski~\cite{nemirovski1995information} reported that BL significantly outperformed subgradient descent on MAXQUAD. Since MAXQUAD has a PWS objective, our analysis provides theoretical justification for this observation.

A closely related approach is the proximal bundle method, introduced in the 1970s~\cite{lemarechal2009extension,mifflin1977algorithm,wolfe2009method}. It enhances stability by adding a proximal regularization term, typically penalizing deviation from the current iterate, to the cutting plane model. Extensions handle inexact oracles~\cite{de2014convex} and non-convex objectives~\cite{hare2010redistributed,de2019proximal}. Its complexity has been studied in~\cite{kiwiel2000efficiency,du2017rate,diaz2023optimal,liang2021proximal}, but the resulting bounds at best match those of subgradient-type methods (e.g., $O(1/\epsilon^2)$ for the general non-smooth convex case) and typically do not capture the benefits of using multiple cuts simultaneously, a feature often credited with the method's strong empirical performance. Exploiting the $\mathcal{VU}$ structure near an optimal solution can yield superlinear convergence~\cite{mifflin2005algorithm,mifflin2012science}, although the resulting guarantees are local.

\paragraph{Structured Non-smooth Optimization.}

Many methods exploit known problem structure to overcome the limitations of general non-smooth optimization. For example, composite methods in sparse optimization handle objectives consisting of a smooth term and a non-smooth regularizer with an efficiently computable proximal mapping~\cite{beck2009fast,nesterov2013gradient}. In function-constrained optimization and risk minimization, prox-linear methods, which are closely related to bundle methods, exploit known max-of-smooth representations~\cite{nesterov2003introductory,Lan15Bundle,zhang2022solving,zhang2023distributed}.

When the Fenchel conjugate is accessible, another approach is to construct a smooth approximation of the objective, as in Nesterov's smoothing method~\cite{Nestrov2004Smooth}. Related constructions include the Moreau envelope~\cite{Beck2017First} and randomized smoothing~\cite{duchi2012randomized}. However, solving the resulting smooth problems may require costly subproblems or involve large Lipschitz constants for the gradients. These costs can prevent smoothing methods from matching the oracle complexity of smooth optimization.

Another line of work exploits local geometric structure near an optimal solution, even when this structure is not explicitly known. These methods rely on a manifold along which the objective behaves smoothly. The $\mathcal{VU}$ framework~\cite{mifflin2005algorithm,mifflin2000functions} provides one such setting: Mifflin and Sagastiz\'abal~\cite{mifflin2005algorithm} combine proximal bundle steps with Newton-like $\mathcal{U}$-steps to achieve local superlinear convergence. Han and Lewis~\cite{han2023survey} develop a related approach, survey descent, that maintains points representing the local smooth regions. Each iteration updates these points by solving $k$ quadratically constrained quadratic programs. Under suitable regularity assumptions, this yields local linear convergence for strongly convex max-of-smooth objectives. The main initialization challenge is to ensure that the survey represents all relevant regions from the outset. For a broader class of non-smooth problems, Davis and Jiang~\cite{davis2024local} propose normal tangent descent. The method exploits suitable active-manifold structure to compute descent directions from Goldstein subgradients. This geometric information enables local nearly linear convergence.

\paragraph{Leveraging Growth Conditions.}
Growth conditions provide a complementary way to identify settings where fast convergence is possible. A condition of order $p$ relates the optimality gap to the distance to the solution set $X^*$ through $\mu \text{dist}^p(x,X^*) \le f(x)-f^*$. Such conditions are closely related to error bounds~\cite{pang1997error} and the Polyak-Lojasiewicz (PL) and Kurdyka-Lojasiewicz (KL) conditions~\cite{polyak1963gradient,lojasiewicz1963propriete,kurdyka1998gradients}. Sharp growth ($p=1$) and quadratic growth (QG, $p=2$) are two important special cases.

Polyak showed that gradient descent converges linearly under the PL condition, without requiring strong convexity~\cite{polyak1963gradient}. For non-smooth convex objectives, his later work established linear convergence of the subgradient method with Polyak stepsizes under sharp growth~\cite{polyak1969minimization}. More recent work exploits these conditions in adaptive algorithms and broader settings~\cite{renegar2022simple,diaz2023optimal,drusvyatskiy2018error}. Our work focuses on the convex PWS setting under QG. By analyzing progress across multiple BL iterations, we establish linear convergence despite non-smoothness.

\subsection{Notation and Assumptions}

\paragraph{Notation}
\begin{itemize}
    \item $[m] := \{0,1,2,\dots,m\}$ denotes the set of integers from 0 to $m$.
    \item $B(x;\delta) := \{y\in \mathbb{R}^n \mid \|y-x\|\leq \delta\}$ denotes the closed ball in $\mathbb{R}^n$ centered at $x$ with radius $\delta \ge 0$.
    \item $f'(x)$ denotes an arbitrary subgradient from the subdifferential $\partial f(x)$. If $f$ is differentiable at $x$, $\nabla f(x)$ denotes its unique gradient.
    \item $l_f(x;y)$ denotes the linear model of $f$ constructed at point $y$ and evaluated at point $x$:
          \[
              l_f(x;y) := f(y)+\langle f'(y), x-y \rangle.
          \]
\end{itemize}

\paragraph{Assumptions}
Throughout the paper, we make the following standard assumptions about the objective function $f$.
\begin{itemize}
    \item We assume $f$ is $M$-Lipschitz continuous for some $M \ge 0$, i.e., $|f(x)-f(y)| \le M\|x-y\|$ for all $x,y \in X$.
    \item We assume $f$ is bounded below, i.e., $f^* := \min_{x\in X} f(x) > -\infty$, and use $\Delta_f$ to denote the initial function value gap, i.e., $\Delta_f:=f(\xt{0})-\fstar$.
    \item We assume the set of minimizers, $X^* := \argmin_{x\in \mathcal{X}}f(x)$, is non-empty. We use $x^*$ to denote an arbitrary optimal solution from this set.
\end{itemize}

\section{The Warm-Up: When $\fstar$ Is Known\protect\label{sec:fstar}}

We illustrate the key idea behind global linear convergence in the simple setting where the optimal objective value $\fstar$ is known. Subsection~\ref{subsec:fstarknown-set-up} describes the problem setting, and Subsection~\ref{subsec:fstar-known-algorithm} presents the bundle-level (BL) algorithm and explains the key insight. The precise convergence guarantees are stated in Subsection~\ref{subsec:fstar-guarantee}, with detailed proofs in Subsection~\ref{subsec:fstar-known-proofs}.

\subsection{The Problem Setup\protect\label{subsec:fstarknown-set-up}}

Throughout this section, we assume that the objective function $f$ in~\eqref{eq:opt_prob} is convex and $(k,L)$-piecewise smooth and satisfies the quadratic growth (QG) condition with modulus $\mu>0$. The definitions are provided below.

\begin{definition}[Piecewise Smoothness]\label{def:piecewise-smooth}
    For $L>0$, we call a function $f:X\rightarrow\mathbb{R}$ $(k,L)$-piecewise smooth if there exists a covering $\{X_i\}_{i=1}^k$ of its domain $X$ such that each component function $f_i := f|_{X_i}$ is $L$-smooth, with a first-order oracle returning a subgradient $f'(\bar{x})\in\partial f(\bar{x})$ satisfying
    \begin{equation}
        f(x)-f(\bar{x})-\inner{f'(\bar{x})}{x-\bar{x}}\leq\frac{L}{2}\norm{x-\bar{x}}^{2}, \quad \forall x,\bar{x}\in X_{i}, \quad \forall i\in\{1,\ldots,k\}. \label{eq:smoothness-condition}
    \end{equation}
\end{definition}

Definition~\ref{def:piecewise-smooth} includes $k$-max-of-smooth functions $f(x):=\max_{i=1,\dots,k} \tilde{f}_{i}(x)$ as a special case, with covering sets $X_{i}=\{x\in X \mid \tilde{f}_{i}(x)=f(x)\}$ determined by which component function attains the maximum. Its oracle requirement is stronger than access to an arbitrary subgradient: at differentiable points, $f'(x)=\nabla f(x)$, whereas at non-smooth boundary points, the returned subgradient must satisfy~\eqref{eq:smoothness-condition}.\footnote{For instance, consider $f(x)=|x|$ on $X=\mathbb{R}$ with the partition $X_{1}=(-\infty,0]$ and $X_{2}=(0,\infty)$. For the region $X_1$, condition~\eqref{eq:smoothness-condition} at $\bar{x}=0$ requires the choice $f'(0)=-1$ from $\partial f(0)=[-1,1]$.}
We retain this assumption for simplicity in this section, but it can be relaxed in the subsequent analysis, for example by ensuring that iterates lie at differentiable points almost surely.

\begin{definition}[Quadratic Growth]\label{def:quadartic_growth} %
    We say that the objective function $f$ satisfies the quadratic growth condition with modulus $\mu>0$ if the following inequality holds:\footnote{Here, $\operatorname{dist}(x,X^{*}) := \inf_{y \in X^*} \|x-y\|$ denotes the Euclidean distance from the point $x$ to the solution set $X^*$.}
    \[
        f(x)-\fstar \geq \frac{\mu}{2} \operatorname{dist}^{2}(x,X^{*}), \quad \forall x\in X.
    \]
\end{definition}
The QG condition is weaker than strong convexity (see~\cite{karimi2016linear}).

\subsection{The Algorithm and the Key Idea \protect\label{subsec:fstar-known-algorithm}}

\begin{algorithm}[ht]
    \caption{The Bundle-Level Method with Known $\fstar$, BL($m,\fstar,x^{0}$)}
    \label{alg:blm-fstarknown}
    \begin{algorithmic}[1]
        \Require Optimal value $\fstar$; initial point $x^{0}\in X$; maximum number of retained cuts $m \ge 1$.
        \For{$t=0,1,2,\dots$}
        \State Define
        \begin{align*}
            X(t) := \{x\in X \mid \inner{f'(x^{t-i})}{x-x^{t-i}}+f(x^{t-i})\leq \fstar, \quad \forall i\in\{0,\ldots,\min\{m-1,t\}\}\}.
        \end{align*}
        \State Compute
        $x^{t+1} \leftarrow \argmin_{x\in X(t)} \frac{1}{2}\norm{x-x^t}^{2}.$ \label{alg:line_4-m}
        \EndFor
    \end{algorithmic}
\end{algorithm}

The bundle-level (BL) method is presented in Algorithm~\ref{alg:blm-fstarknown}. At each iteration $t$, the next iterate $x^{t+1}$ is obtained by projecting the current iterate $x^{t}$ onto the $\fstar$-sublevel set $X(t)$ of the model function $\max_{i\in[\min\{m-1,t\}]} l_f(x;\xt{t-i})$. The method follows the classical bundle-level scheme~\cite{lemarechal1995new}, with the level parameter fixed at $\fstar$ and memory limited to the $m$ most recent cuts.

For PWS optimization, the BL method makes significant progress when the next iterate lands in the same region as an earlier iterate and the earlier cut remains in the bundle to provide useful first-order information. The following definition formalizes such a pair of matching iterates.

\begin{definition}[$l$-Matching Pair]\label{def:mathcing-pair}
    Given a trajectory $\{x^t\}_{t \ge 0}$, a pair $(x^{t}, x^{t+j})$ is called an $l$-matching pair with respect to the regions $\{X_i\}_{i=1}^k$ if $1\leq j\leq l$ and $\{x^{t}, x^{t+j}\}\subseteq X_{\bar{\imath}}$ for some $\bar{\imath}\in\{1,...,k\}$.
\end{definition}

The parameter $l$ bounds the separation between the two iterates. If $l\leq m$, the earlier cut remains in the bundle when the later iterate is computed. The following lemma extends the stability property of consecutive level projections, expressed by the three-point inequality~\cite{Beck2017First,LanBook}, to non-consecutive iterates.

\begin{lemma}\label{lm:bridged-three-point-euclid}
    If the iterates $\{x^{t}\}_{t\ge 0}$ satisfy $\inner{x^{t+1}-x^{t}}{x^{t+1}-x^{*}} \leq 0$ for all $\xstar \in X^*$ and $t \ge 0$, then the following \textbf{bridged three-point inequality} holds:
    \[
        \|x^{t+j}-x^{*}\|^{2}+\frac{1}{j}\|x^{t+j}-x^t\|^{2} \leq \|x^t - x^{*}\|^{2}, \quad \forall t\geq 0, j\geq1, \xstar\in X^*.
    \]
    In particular, the iterates generated by Algorithm~\ref{alg:blm-fstarknown} satisfy this inequality for every $x^*\in X^*$.
\end{lemma}

\begin{proof}
    We first verify the stated condition for Algorithm~\ref{alg:blm-fstarknown}. By convexity, every cut retained at iteration $t$ satisfies
    \[
        f(x^{t-i})+\inner{f'(x^{t-i})}{x^*-x^{t-i}}\leq f(x^*)=\fstar, \quad \forall x^*\in X^*.
    \]
    Hence $X^*\subseteq X(t)$, and the projection defining $x^{t+1}$ yields
    \[
        \inner{x^{t+1}-x^t}{x^{t+1}-x^*}\leq 0, \quad \forall x^*\in X^*.
    \]
    Fix any $\xstar \in X^*$. The standard three-point inequality~\cite{LanBook} gives
    \begin{equation}\label{fstareq:monotonicity}
        \|x^{t+1}-x^*\|^2+\|x^{t+1}-x^t\|^2\leq\|x^t-x^*\|^2, \quad \forall t\geq 0.
    \end{equation}
    Summing~\eqref{fstareq:monotonicity} over iterations $t,\dots,t+j-1$ yields
    \[
        \|x^{t+j}-x^*\|^2+\sum_{i=1}^{j}\|x^{t+i}-x^{t+i-1}\|^2\leq\|x^t-x^*\|^2.
    \]
    By the Cauchy--Schwarz inequality,
    \[
        \|x^{t+j}-x^t\|^2
        =\left\|\sum_{i=1}^{j}(x^{t+i}-x^{t+i-1})\right\|^2
        \leq j\sum_{i=1}^{j}\|x^{t+i}-x^{t+i-1}\|^2.
    \]
    Combining these two inequalities proves the claim. \qedsymbol
\end{proof}

The following proposition uses the bridged three-point inequality to establish a contraction for each $k$-matching pair.

\begin{proposition}[Contraction at a Matching Pair]\label{pr:fstar-matching-contraction}
    Assume that $f$ is convex and $(k,L)$-piecewise smooth, satisfies the quadratic growth condition with modulus $\mu>0$, and has a nonempty optimal set $X^*$. Let $\{x^t\}_{t\geq 0}$ be generated by Algorithm~\ref{alg:blm-fstarknown} with $m\geq k$. Then every $k$-matching pair $(x^t,x^{t+j})$ satisfies
    \begin{equation}\label{fstareq:simple-argument}
        \operatorname{dist}^{2}(x^{t+j},X^*) \leq \frac{kL}{kL+\mu}\operatorname{dist}^{2}(x^t,X^*).
    \end{equation}
\end{proposition}

\begin{proof}
    Let $(x^t,x^{t+j})$ be a $k$-matching pair in the region $X_{\bar{\imath}}$. Since $j\leq k\leq m$, the cut generated at $x^t$ is retained in $X(t+j-1)$. The feasibility of $x^{t+j}$ therefore gives
    \[
        f(x^t)+\inner{f'(x^t)}{x^{t+j}-x^t}\leq\fstar.
    \]
    Combining this inequality with smoothness on $X_{\bar{\imath}}$ and the quadratic growth condition yields
    \[
        \frac{\mu}{2}\operatorname{dist}^{2}(x^{t+j},X^*)
        \leq f(x^{t+j})-\fstar
        \leq\frac{L}{2}\|x^{t+j}-x^t\|^2.
    \]
    Substituting the resulting lower bound on $\|x^{t+j}-x^t\|^2$ into Lemma~\ref{lm:bridged-three-point-euclid} and using $\|x^{t+j}-x^*\|^2\geq\operatorname{dist}^{2}(x^{t+j},X^*)$ gives
    \begin{equation}\label{pfeq:fstar-pari-contraction}
        \left(1+\frac{\mu}{jL}\right)\operatorname{dist}^{2}(x^{t+j},X^*)\leq\|x^t-x^*\|^2, \quad \forall x^*\in X^*.
    \end{equation}
    Taking the infimum over $x^*\in X^*$ on the right-hand side and using $j\leq k$ yields~\eqref{fstareq:simple-argument}. \qedsymbol
\end{proof}

Since the pigeonhole principle guarantees a $k$-matching pair within every $k+1$ consecutive iterates, the proposition above ensures a significant contraction towards $X^*$ within each such block. This yields the desired global linear convergence of $\operatorname{dist}(x^t,X^*)$ to zero when $m\geq k$.

\subsection{A More Precise Convergence Guarantee}\label{subsec:fstar-guarantee}

A closer inspection of \eqref{pfeq:fstar-pari-contraction} reveals that the contraction depends on the actual length $j$ of the matching pair $(x^t,x^{t+j})$, rather than the conservative upper bound $k$. Thus, for a given trajectory of the BL method, a more precise convergence guarantee should account for both the average length of matching pairs and how frequently they occur. The next definition formalizes these statistics.

\begin{definition}[$l$-Matching Pair Sequence and Statistics]\label{def:matching-sequence}
    We call $\{(\xt{l_i},\xt{r_i})\}$ an $l$-matching pair sequence if each $(x^{l_i},x^{r_i})$ is an $l$-matching pair (see Definition~\ref{def:mathcing-pair}) and the index intervals do not overlap, i.e., $r_i\leq l_{i+1}<r_{i+1}$ for all $i$. For an iteration limit $N\geq 1$, let $P(N):=\max\bigl(\{0\}\cup\{i\geq 1\mid r_i\leq N\}\bigr)$ denote the number of completed matching pairs. When $P(N)\geq 1$, we define the \emph{average inter-arrival time} $\taubar(N)$ and the \emph{average length} $\sigbar(N)$ as
    \begin{equation}
        \taubar(N):=\frac{N}{P(N)}, \quad \sigbar(N):= P(N) \left/ \sum_{i=1}^{P(N)}\frac{1}{r_{i}-l_{i}} \right. .
        \label{eq:matching-separation-distance} %
    \end{equation}

\end{definition}

For any $(k,L)$-piecewise smooth function, the pigeonhole principle implies that there always exists some $k$-matching pair sequence with $\taubar(N)\leq 2k$ and $\sigbar(N)\leq k+1$. %

Now we are ready to state the formal convergence guarantee.

\begin{theorem}\label{thm:fstar-convergence}
    For problem~\eqref{eq:opt_prob}, assume $f$ is convex and $(k,L)$-piecewise smooth (Definition~\ref{def:piecewise-smooth}). Consider the iterates $\{x^t\}_{t\ge 0}$ generated by Algorithm~\ref{alg:blm-fstarknown} with $m\geq k$ cuts, optimal value $\fstar$, and initial point $x^{0}\in X$. Let $\{(x^{l_i}, x^{r_i})\}$ be any $k$-matching pair sequence (Definition~\ref{def:matching-sequence}) along these iterates, with statistics $\taubar(N)$ and $\sigbar(N)$ defined in~\eqref{eq:matching-separation-distance}. The following bounds hold for $N\geq 1$ with $P(N)\geq 1$:

    a) The function values satisfy
    \[
        \min_{t \in \{1, \dots, N\}} f(x^{t})-\fstar \leq \frac{L\sigbar(N)\taubar(N)}{2N} \operatorname{dist}^{2}(x^{0},X^{*}).
    \]

    b) Moreover, if $f$ satisfies the quadratic growth condition (Definition~\ref{def:quadartic_growth}) %
    with modulus $\mu>0$, then
    \begin{align*}
        \operatorname{dist}^{2}(x^{N},X^{*})      & \leq 2\left(1+\frac{1}{\kappa(N)}\right)^{-N} \operatorname{dist}^{2}(x^{0},X^{*}), \\
        \min_{t\in\{0, \dots, N\}}f(x^{t})-\fstar & \leq 2\left(1+\frac{1}{\kappa(N)}\right)^{-N} [f(x^{0})-\fstar],
    \end{align*}
    where
    \[
        \kappa(N)=\sigbar(N)\taubar(N)\left(1+\frac{2L}{\mu}\right).
    \]
\end{theorem}

\begin{proof}
    We defer the detailed analysis to Subsection~\ref{subsec:fstar-known-proofs}.
\end{proof}

We make a few remarks regarding Theorem~\ref{thm:fstar-convergence}. For the complexity comparisons below, we assume $L\geq\mu$ for simplicity.

First, the complexity bound depends on the matching pair statistics $\sigbar(N)$ and $\taubar(N)$, which reflect actual iterate behavior and can provide a more precise guarantee than the total number of regions $k$ alone. In particular, Theorem~\ref{thm:fstar-convergence}(b) yields the following iteration complexity $N(\epsilon)$ to reach $\epsilon$-accuracy:
\[
    N(\epsilon) = O\left(\frac{L\sigbar(N(\epsilon))\taubar(N(\epsilon))}{\mu}\log\left(\frac{1}{\epsilon}\right)\right).
\]
Here, $L/\mu$ is the condition number familiar from non-accelerated methods for strongly convex smooth optimization. The factors $\sigbar(N)$ and $\taubar(N)$ capture the additional cost of the unknown PWS structure (see Definition~\ref{def:matching-sequence}). The pigeonhole principle guarantees a $k$-matching pair sequence with $\taubar(N)\leq O(k)$ and $\sigbar(N)\leq k$, giving the worst-case oracle complexity bound
\[
    N(\epsilon)\leq O\left(\frac{Lk^{2}}{\mu}\log\left(\frac{1}{\epsilon}\right)\right).
\]
For comparison, consider a $k$-max-of-smooth objective $f(x) = \max_{i=1,\dots,k}\tilde{f}_{i}(x)$. With oracle access to each smooth component $\tilde{f}_i$, the unaccelerated prox-linear method achieves $O\left(\frac{Lk}{\mu}\log\left(\frac{1}{\epsilon}\right)\right)$ complexity~\cite{nesterov2003introductory}. This comparison suggests that our algorithm incurs an additional factor related to $k$ because the individual regions are unknown. However, smaller values of $\sigbar(N)$ and $\taubar(N)$ may yield better practical performance than the $O(k^2)$ worst-case bound suggests.

Second, Algorithm~\ref{alg:blm-fstarknown} does not require the quadratic growth modulus $\mu$ or smoothness constant $L$, a property common to bundle-level methods~\cite{Lan15Bundle,lemarechal1995new}. Its main inputs are the optimal value $\fstar$ and the number of cuts $m$, with $m \ge k$ required for the linear convergence guarantee in Theorem~\ref{thm:fstar-convergence}. As discussed in Contribution~3, Sections~\ref{sec:QG} and~\ref{sec:pf} develop methods that remove the dependence on $\fstar$. The memory parameter $m$ controls the number of stored cuts and balances iteration complexity against the cost of solving the quadratic program at each iteration (Line 3 of Algorithm~\ref{alg:blm-fstarknown}).

Third, the choice of $m$ affects convergence. While $m \ge k$ is assumed for the global linear rate in Theorem~\ref{thm:fstar-convergence}, choosing a smaller $m$ (e.g., based on the number of regions $k_{\text{local}}$ near $x^*$) reduces the cost per iteration. The analysis in Section~\ref{sec:inexact} shows that with such smaller $m$, the algorithm typically first achieves an $O(1/N)$ sublinear rate, followed by local linear convergence once the iterates enter the relevant neighborhood. Furthermore, under a sharper growth condition than quadratic growth, such as $f(x)-\fstar\geq\mu\operatorname{dist}(x,X^{*})$, Algorithm~\ref{alg:blm-fstarknown} might achieve faster local convergence, potentially quadratic, improving over the linear guarantee for the proximal bundle method in~\cite{diaz2023optimal}.

\subsection{The Detailed Convergence Proof\protect\label{subsec:fstar-known-proofs}}

We first collect two useful consequences of the preceding projection and matching-pair arguments.

\begin{corollary}[Iterate Properties]\label{cor:fstar-iterate-properties}
    In the setting of Theorem~\ref{thm:fstar-convergence}, the following properties hold:

    a) The distance to $X^*$ is non-increasing:
    \[
        \|x^{t+1}-x^*\|^2\leq\|x^t-x^*\|^2, \quad \forall t\geq 0,\; x^*\in X^*.
    \]

    b) Every $k$-matching pair $(x^l,x^r)$ satisfies
    \begin{equation}\label{eq:fstar-matching-progress}
        f(x^r)-\fstar\leq\frac{L(r-l)}{2}\left(\|x^l-x^*\|^2-\|x^r-x^*\|^2\right), \quad \forall x^*\in X^*.
    \end{equation}
\end{corollary}

\begin{proof}
    Part (a) follows from Lemma~\ref{lm:bridged-three-point-euclid} with $j=1$. For part (b), the cut at $x^l$ is retained when $x^r$ is computed, since $r-l\leq k\leq m$. As in the proof of Proposition~\ref{pr:fstar-matching-contraction}, feasibility and smoothness on the common region give
    \begin{equation}\label{eq:fval-prog-user}
        f(x^r)-\fstar\leq\frac{L}{2}\|x^r-x^l\|^2.
    \end{equation}
    Combining this bound with Lemma~\ref{lm:bridged-three-point-euclid}, applied with $t=l$ and $j=r-l$, proves part (b). \qedsymbol
\end{proof}

We now combine these properties over the completed matching pairs to prove Theorem~\ref{thm:fstar-convergence}(a).

\begin{proof}[Proof of Theorem~\ref{thm:fstar-convergence}(a)]
    Let $x^*\in X^*$ be an arbitrary optimal solution. Applying Corollary~\ref{cor:fstar-iterate-properties}(b) to the completed matching pairs, summing over $i=1,\dots,P(N)$, and using part (a) of the corollary to telescope yields
    \[
        \sum_{i=1}^{P(N)}\frac{2}{r_{i}-l_{i}}(f(x^{r_i})-\fstar) \leq L\|x^{l_1}-x^*\|^2 \leq L\|x^0-x^*\|^2.
    \]
    Let $f_{\min, N} := \min_{t \in \{1, \dots, N\}} f(x^t)$. Lower bounding $f(x^{r_i})$ by $f_{\min, N}$:
    \[
        \left(\sum_{i=1}^{P(N)}\frac{2}{r_{i}-l_{i}}\right) (f_{\min, N} - \fstar) \leq \sum_{i=1}^{P(N)}\frac{2}{r_{i}-l_{i}}(f(x^{r_i})-\fstar) \leq L\|x^0-x^*\|^2.
    \]
    By definition~\eqref{eq:matching-separation-distance} (see also Definition~\ref{def:matching-sequence}), we have $\sum_{i=1}^{P(N)}\frac{1}{r_{i}-l_{i}} = P(N)/\sigbar(N)$ and $P(N) = N/\taubar(N)$. Therefore, the sum is $\frac{N}{\taubar(N)\sigbar(N)}$. Substituting this gives:
    \[
        \frac{2N}{\taubar(N)\sigbar(N)} (f_{\min, N} - \fstar) \leq L\|x^0-x^*\|^2.
    \]
    Rearranging and minimizing the right-hand side over $x^* \in X^*$ (yielding $\operatorname{dist}^2(x^0, X^*)$) gives:
    \begin{equation}
        \min_{t\in\{1, \dots, N\}}[f(x^t)-\fstar] \leq \frac{L\taubar(N)\sigbar(N)}{2N}\operatorname{dist}^{2}(x^{0},X^{*}). \label{pfeq:fstar-known-function-val-detailed}
    \end{equation}
    \qedsymbol
\end{proof}

We prove part (b) by amortizing the progress over the matching pairs.

\begin{proof}[Proof of Theorem~\ref{thm:fstar-convergence}(b)]
    Define the cumulative weights
    \[
        A_0:=1, \qquad
        A_i:=\prod_{h=1}^{i}\left(1+\frac{\mu}{2L(r_h-l_h)}\right),
        \quad i=1,\dots,P(N).
    \]
    For each matching pair, rearranging Corollary~\ref{cor:fstar-iterate-properties}(b) and taking the infimum over $x^*\in X^*$ gives
    \[
        \dist^2(x^{r_i},X^*)+\frac{2[f(x^{r_i})-\fstar]}{L(r_i-l_i)}
        \leq\dist^2(x^{l_i},X^*).
    \]
    Split the function-gap term into two equal parts and apply the QG condition to one of them. Multiplying by $A_{i-1}$ then yields
    \[
        A_i\dist^2(x^{r_i},X^*)
        +\frac{2}{\mu}(A_i-A_{i-1})[f(x^{r_i})-\fstar]
        \leq A_{i-1}\dist^2(x^{l_i},X^*).
    \]
    Since $r_i\leq l_{i+1}$ and Corollary~\ref{cor:fstar-iterate-properties}(a) ensures that the distance to $X^*$ is non-increasing, summing over the completed pairs telescopes to
    \begin{equation}\label{eq:fstar-amortized}
        A_{P(N)}\dist^2(x^N,X^*)
        +\frac{2}{\mu}\sum_{i=1}^{P(N)}(A_i-A_{i-1})[f(x^{r_i})-\fstar]
        \leq\dist^2(x^0,X^*).
    \end{equation}
    Dropping the sum gives the distance bound with factor $A_{P(N)}^{-1}$. For the function-value bound, use $\sum_{i=1}^{P(N)}(A_i-A_{i-1})=A_{P(N)}-1$, the QG condition at $x^0$, and the inclusion of $x^0$ among the candidate iterates to obtain
    \[
        \min_{0\leq t\leq N}[f(x^t)-\fstar]
        \leq\min\left\{1,\frac{1}{A_{P(N)}-1}\right\}[f(x^0)-\fstar]
        \leq\frac{2}{A_{P(N)}}[f(x^0)-\fstar].
    \]

    It remains to express the weights in terms of the matching pair statistics. Since $r_i-l_i\geq 1$, concavity of the logarithm and $\log(1+u)\geq u/(1+u)$ for $u\geq 0$ imply
    \[
        \log A_{P(N)}
        \geq\log\left(1+\frac{\mu}{2L}\right)\sum_{i=1}^{P(N)}\frac{1}{r_i-l_i}
        \geq\frac{N}{\sigbar(N)\taubar(N)(1+2L/\mu)}
        =\frac{N}{\kappa(N)}.
    \]
    Thus $A_{P(N)}^{-1}\leq e^{-N/\kappa(N)}\leq(1+1/\kappa(N))^{-N}$, proving both bounds. \qedsymbol
\end{proof}

\section{Approximately Piecewise Smooth Function \protect\label{sec:inexact}}

In this section, we introduce a more general class of approximately piecewise smooth functions (see Definition~\ref{def:apx-p.w.s}). This generalization allows us to use a standard subgradient oracle and extends the method's applicability to a wider range of non-smooth functions. We show that, with minor modifications, the proposed bundle-level method can still achieve linear convergence for this class.

\subsection{The Problem Set-up}

\begin{definition}[Approximate Piecewise Smoothness]\label{def:apx-p.w.s}
    For $L\geq 0$ and $\delta\geq 0$, we call a convex function $f:X\rightarrow\mathbb{R}$ \emph{$(k,L,\delta)$-approximately piecewise smooth} (apx-PWS) if there exists a covering $\{X_i\}_{i=1}^k$ of its domain $X$, with an oracle returning at each query $\bar{x}\in X$ an affine minorant $\tilde{l}_f(\cdot;\bar{x})$ of $f$ on $X$ satisfying
    \[
        f(x)-\tilde{l}_f(x;\bar{x})\leq\frac{L}{2}\|x-\bar{x}\|^2+\delta, \quad \forall x,\bar{x}\in X_i, \quad \forall i\in\{1,\ldots,k\}.
    \]
\end{definition}

Clearly, a $(k,L)$-PWS function (see Definition~\ref{def:piecewise-smooth}) is also a $(k,L,0)$-apx-PWS function. Let us motivate this definition further by discussing two important applications.

\begin{itemize}
    \item \textit{Using a standard first-order oracle:}
          Definition~\ref{def:piecewise-smooth} requires a specific subgradient choice to satisfy the smoothness condition~\eqref{eq:smoothness-condition} within each region $X_i$. For a convex, $M$-Lipschitz, $(k,L)$-PWS function $f:\mathbb{R}^{d}\rightarrow\mathbb{R}$, we can use a standard first-order oracle by perturbing the query point. At each query $\bar{x}\in X$, we sample $\tilde{x}$ uniformly from the Euclidean ball $B(\bar{x}; \bar{\delta})$ and define
          \begin{equation}
              \tilde{l}_f(x;\bar{x}) := f(\tilde{x})+\inner{\nabla f(\tilde{x})}{x-\tilde{x}}.
              \label{eq:linear-functional-apprx}
          \end{equation}
          By Rademacher's theorem, $f$ is differentiable at $\tilde{x}$ with probability one, and convexity ensures that $\tilde{l}_f$ is an affine minorant of $f$.

          To verify approximate smoothness, fix $\delta>0$ and choose $\bar{\delta}\leq\min\{\sqrt{\delta/(8L)}, \delta/(4M)\}$. For any finite sequence of evaluation points $\{x^t\}_{t=1}^N$\footnote{Requiring the approximate smoothness bound only between queried iterates is weaker than Definition~\ref{def:apx-p.w.s}, but suffices for the convergence analysis. Region membership is assigned to each query through its perturbed point, so repeated queries at the same point may have different assignments.}, assign $x^t$ to region $\tilde{X}_i$ if its perturbed point $\tilde{x}^t$ belongs to the original region $X_i$. These derived regions cover these evaluation points. For two points in the same derived region $\tilde{X}_{\bar{\imath}}$, Lemma~\ref{lem:approx-smooth-transfer} gives
          \[
              f(x^i)-\tilde{l}_f(x^i; x^j) = f(x^i)- l_f(x^i; \tilde{x}^j) \leq L\|x^i - x^j\|^2 + \delta, \quad \forall x^i, x^j \in \tilde{X}_{\bar{\imath}}.
          \]
          Thus, randomly perturbing the query points before evaluating subgradients ensures that $f$ is $(k,2L,\delta)$-apx-PWS with respect to $\{x^t\}_{t=1}^N$.

    \item \textit{General Lipschitz continuous (non-smooth) convex functions:} Any $M$-Lipschitz continuous convex function $f$ is $(1,2M^{2}/\delta,\delta)$-apx-PWS for any $\delta>0$. Specifically, the  Young's inequality gives
          \begin{equation}\label{eq:lipschitz-apx-pws}
              f(x)-l_f(x;\xbar)\leq 2M\|x-\xbar\|\leq\frac{M^{2}}{\delta}\|x-\xbar\|^{2}+\delta, \quad \forall x,\xbar\in X.
          \end{equation}
          More generally, convex functions with H\"older-continuous gradients are $(1,L,\delta)$-apx-PWS for any $\delta>0$ and a suitable $L$, see ~\cite{nesterov2015universal} for details.
\end{itemize}

\subsection{The Algorithm and the Convergence Guarantee}

\begin{algorithm}[ht] %
    \caption{The Approximate Bundle Level Method with Known $\fstar$, apx-BL($m,\fstar,x^{0}$)} %
    \label{alg:blm-fstarknown-approx} %

    \begin{algorithmic}[1] %
        \Require Optimal value $\fstar$; initial point $x^{0}\in X$; number of cuts parameter $m \ge 1$. %
        \For{$t=0,1,2,\dots$} %
        \State Define the level set using the $m$ most recent affine minorants: %
        \begin{align*}
            X(t) := \{x\in X \mid \tilde{l}_f(x; x^{t-i}) \leq \fstar, \quad \forall i \in \{0,\dots,\min\{m-1,t\}\} \}.
        \end{align*}
        \State Compute the next iterate by projection: %
        $x^{t+1} \leftarrow \argmin_{x\in X(t)} \frac{1}{2}\|x-x^t\|^{2}.$ %
        \EndFor
    \end{algorithmic}
\end{algorithm}

The approximate bundle level method (apx-BL, Algorithm~\ref{alg:blm-fstarknown-approx}) replaces the exact linearizations in Algorithm~\ref{alg:blm-fstarknown} with the affine minorants $\tilde{l}_f(\cdot;x^i)$ from Definition~\ref{def:apx-p.w.s}. It reduces to the original method when these minorants are the standard linearizations. The matching-pair framework from the previous section extends to this setting and yields the following convergence guarantee.

\begin{theorem}\label{thm:fstar-convergence-approx}
    For problem~\eqref{eq:opt_prob}, assume $f$ is convex and $(k,L,\delta)$-approximately piecewise smooth (Definition~\ref{def:apx-p.w.s}), with $L>0$. Consider the iterates $\{x^t\}_{t\ge 0}$ generated by Algorithm~\ref{alg:blm-fstarknown-approx} with $m\geq k$ cuts, optimal value $\fstar$, and initial point $x^{0}\in X$. Let $\{(x^{l_i}, x^{r_i})\}$ be any $k$-matching pair sequence (Definition~\ref{def:matching-sequence}) along these iterates, with statistics $\taubar(N)$ and $\sigbar(N)$ defined in~\eqref{eq:matching-separation-distance}. The following bounds hold for $N\geq 1$ with $P(N)\geq 1$:

    a) The function values satisfy
    \[
        \min_{t \in \{1, \dots, N\}} f(x^{t})-\fstar \leq \frac{L\sigbar(N)\taubar(N)}{2N} \operatorname{dist}^{2}(x^{0},X^{*}) + \delta.
    \]

    b) Moreover, if $f$ satisfies the quadratic growth condition (Definition~\ref{def:quadartic_growth})
    with modulus $\mu>0$, then
    \begin{align*}
        \operatorname{dist}^{2}(x^{N},X^{*})      & \leq 2\left(1+\frac{1}{\kappa(N)}\right)^{-N} \operatorname{dist}^{2}(x^{0},X^{*}) + \frac{4\delta}{\mu}, \\
        \min_{t\in\{0, \dots, N\}}f(x^{t})-\fstar & \leq 2\left(1+\frac{1}{\kappa(N)}\right)^{-N} [f(x^{0})-\fstar] + 2\delta,
    \end{align*}
    where
    \[
        \kappa(N)=\sigbar(N)\taubar(N)\left(1+\frac{2L}{\mu}\right).
    \]
\end{theorem}

\begin{proof}
    We defer the detailed analysis to Subsection~\ref{subsec:analysis-fstar-approx}. %
\end{proof}

For target accuracies $\epsilon\geq4\delta$, Theorem~\ref{thm:fstar-convergence-approx} gives complexity comparable to that of BL, while apx-BL retains its independence from $L$, $\mu$, and $\delta$.

The conversion~\eqref{eq:lipschitz-apx-pws} ensures convergence for general $M$-Lipschitz convex objectives with any $m\geq1$. Since the same subgradient cuts satisfy this bound for every $\delta>0$, the analysis can choose $\delta$ adaptively as $\delta^t$, for example,
\[
    \delta^t=o(1)\min_{j\in[t]}[f(x^j)-\fstar].
\]
We could then recover the corresponding oracle complexity bounds for non-smooth optimization: $O(M^2\dist^2(x^0,X^*)/\epsilon^2)$ for general convex objectives and $O(M^2/(\mu\epsilon))$ under the QG condition~\cite{nesterov2003introductory}.

Additionally, the preceding universal convergence ensures that the iterates eventually enter a neighborhood of an optimal solution under the QG condition. If this neighborhood contains at most $m$ relevant smooth regions, the matching-pair analysis yields a local linear rate, up to the approximation error.

\subsection{The Detailed Convergence Analysis\protect\label{subsec:analysis-fstar-approx}}

We follow the argument in Subsection~\ref{subsec:fstar-known-proofs}, tracking the approximation error $\delta$. The following corollary extends Corollary~\ref{cor:fstar-iterate-properties} to the apx-PWS setting.

\begin{corollary}[Iterate Properties]\label{lem:fstar-apx-pflm}
    In the setting of Theorem~\ref{thm:fstar-convergence-approx}, the following properties hold:

    a) The distance to every optimal solution is non-increasing:
    \[
        \|x^{t+1}-x^*\|^2\leq\|x^t-x^*\|^2, \quad \forall t\geq 0,\; x^*\in X^*.
    \]

    b) Every $k$-matching pair $(x^l,x^r)$ satisfies
    \begin{equation}\label{eq:fstar-apx-matching-progress}
        f(x^r)-\fstar-\delta\leq\frac{L(r-l)}{2}\left(\|x^l-x^*\|^2-\|x^r-x^*\|^2\right), \quad \forall x^*\in X^*.
    \end{equation}
\end{corollary}

\begin{proof}
    The affine minorants ensure that $X^*\subseteq X(t)$ at every iteration. Hence the projection step satisfies the hypothesis of Lemma~\ref{lm:bridged-three-point-euclid}, which gives part (a).

    For part (b), the cut at $x^l$ is retained when $x^r$ is computed because $r-l\leq k\leq m$. Feasibility and approximate smoothness on the common region give
    \[
        f(x^r)-\fstar-\delta
        \leq f(x^r)-\lftil(x^r;x^l)-\delta
        \leq\frac{L}{2}\|x^r-x^l\|^2.
    \]
    Applying Lemma~\ref{lm:bridged-three-point-euclid} with $t=l$ and $j=r-l$ proves part (b). \qedsymbol
\end{proof}

We first sum this progress bound over the completed matching pairs to prove Theorem~\ref{thm:fstar-convergence-approx}(a).

\begin{proof}[Proof of Theorem~\ref{thm:fstar-convergence-approx}(a)]
    As in the proof of Theorem~\ref{thm:fstar-convergence}(a), Corollary~\ref{lem:fstar-apx-pflm} gives
    \begin{equation}\label{pfeq:del-accum-progress}
        \sum_{i=1}^{P(N)}\frac{2[f(x^{r_i})-\fstar-\delta]}{r_i-l_i}
        \leq L\|x^0-x^*\|^2, \quad \forall x^*\in X^*.
    \end{equation}
    By~\eqref{eq:matching-separation-distance}, $\sum_{i=1}^{P(N)}(r_i-l_i)^{-1}=N/[\sigbar(N)\taubar(N)]$. Bounding each function value below by the minimum over $t=1,\dots,N$, rearranging, and taking the infimum over $x^*\in X^*$ yields
    \[
        \min_{1\leq t\leq N}[f(x^t)-\fstar]
        \leq\frac{L\sigbar(N)\taubar(N)}{2N}\dist^2(x^0,X^*)+\delta.
    \]
    \qedsymbol
\end{proof}

For part (b), we use the same amortization argument until a $2\delta$-optimal iterate is found.

\begin{proof}[Proof of Theorem~\ref{thm:fstar-convergence-approx}(b)]
    If $f(x^t)-\fstar\leq2\delta$ for some $0\leq t\leq N$, the QG condition and Corollary~\ref{lem:fstar-apx-pflm}(a) give $\dist^2(x^N,X^*)\leq4\delta/\mu$, so both bounds hold.

    Otherwise, $f(x^t)-\fstar>2\delta$ for all $0\leq t\leq N$. Form the weights $A_i$ as in the proof of Theorem~\ref{thm:fstar-convergence}(b), using the present matching pairs. Applying the same QG split and telescoping argument to Corollary~\ref{lem:fstar-apx-pflm}(b) gives
    \begin{equation}\label{eq:fstar-apx-amortized}
        A_{P(N)}\dist^2(x^N,X^*)
        +\frac{2}{\mu}\sum_{i=1}^{P(N)}(A_i-A_{i-1})[f(x^{r_i})-\fstar-2\delta]
        \leq\dist^2(x^0,X^*).
    \end{equation}
    The summands are nonnegative, so the remaining steps in the proof of Theorem~\ref{thm:fstar-convergence}(b) yield
    \begin{align*}
        \dist^2(x^N,X^*)
         & \leq\frac{\dist^2(x^0,X^*)}{A_{P(N)}},         \\
        \min_{0\leq t\leq N}[f(x^t)-\fstar]
         & \leq\frac{2}{A_{P(N)}}[f(x^0)-\fstar]+2\delta.
    \end{align*}
    The same logarithmic estimate gives $A_{P(N)}^{-1}\leq e^{-N/\kappa(N)}\leq(1+1/\kappa(N))^{-N}$ with the stated $\kappa(N)$, proving both bounds. \qedsymbol
\end{proof}

\section{Assuming the Knowledge of the Quadratic Growth Parameter}\label{sec:QG}

In this section, we remove the requirement that the optimal objective value $f^*$ be known. For convex problems satisfying the QG condition, we assume a known positive modulus $\mu$ and develop a bundle-level method that maintains upper and lower bounds on $f^*$ and reduces their gap.  We then apply this method within an inexact proximal point scheme to find approximately stationary points of weakly convex problems with a known weak convexity constant $\rho$.

\subsection{The Bundle Level Method with a Known $\mu$.}

\begin{algorithm}
    \caption{The Gap Reduction subroutine ($\GR$)}
    \label{alg:Gap-Reduction}
    \begin{algorithmic}[1]
        \Require The quadratic growth parameter $\mu > 0$; initial point $x^{0}$; initial upper bound $\bar{f}=f(x^{0})$; initial lower bound $\underline{f}$; number of cuts $m$.
        \Ensure Updated iterate $x^{+}$; updated upper bound $\bar{f}^{+}=f(x^{+})$; updated lower bound $\underline{f}^{+}$.

        \State \textbf{Initialize:} $(x^+,\bar{f}^+)\gets(x^0,\bar{f})$, $\Delta_0 \gets \bar{f}-\underline{f}$, $l \gets \underline{f}+\Delta_0/3$, $\tilde{S}_{r}(0) \gets 0$, $\tilde{S}_{l}(0) \gets 0$.
        \State \textbf{if} $\Delta_0=0$, \Return{$(x^+,\bar{f}^+,\underline{f})$}.
        \For{$t=0,1,2,\dots$}
        \parState{Define the level set using the fixed level $l$:
        \[
            X(t+1):=\{x\in X\mid\tilde{l}_f(x;x^{t-j})\leq l,\quad j\in\{0,\dots,\min\{m-1,t\}\}\}.
        \]}
        \State \label{GR:line:infeasible}\textbf{if} $X(t+1)=\emptyset$, \Return{$(x^+,\bar{f}^+,l)$}.
        \State Compute $x^{t+1}\gets\argmin_{x\in X(t+1)}\|x-x^t\|^2$.
        \State \textbf{if} $f(x^{t+1})<\bar{f}^+$, set $(x^+,\bar{f}^+)\gets(x^{t+1},f(x^{t+1}))$.
        \State \label{GR:line:gap-reduction-term}\textbf{if} $\bar{f}^+\leq\bar{f}-\Delta_0/3$, \Return{$(x^+,\bar{f}^+,\underline{f})$}.
        \parState{\label{GR:line:DP} With $\tau\gets\max\{0,t+1-m\}$, update the empirical progress:
            \[
                \begin{aligned}
                    \tilde{S}_{r}(t+1) & \gets\max_{\tau\leq q\leq t}\left\{\tilde{S}_{l}(q)+\frac{1}{(t+1-q)\tilde{L}(x^{t+1},x^q;\Delta_0)}\right\}, \\
                    \tilde{S}_{l}(t+1) & \gets\max\{\tilde{S}_{l}(t),\tilde{S}_{r}(t+1)\}.
                \end{aligned}
            \]
        }
        \If{\label{GR:line:lower-bound-update}$\tilde{S}_{l}(t+1)\geq6/\mu$}
        \State \label{GR:line:lower-bound-update-1}\Return{$(x^+,\bar{f}^+,l)$}.
        \EndIf
        \EndFor
    \end{algorithmic}
\end{algorithm}

\begin{algorithm}
    \caption{The Bundle Level method with a Known $\mu$ (BL-$\mu$)}
    \label{alg:blm-u}
    \begin{algorithmic}[1]
        \Require Initial point $x^{0}$, quadratic growth parameter $\mu > 0$, number of cuts $m$, and target accuracy $\epsilon > 0$.
        \Ensure An $\epsilon$-optimal solution $\hat{x}$.

        \State \textbf{Initialize:} $\bar{f}^{0} \gets f(x^{0})$ and $\underline{f}^{0} \gets f(x^{0}) - 2\|f'(x^{0})\|^{2}/\mu$.

        \For{$s = 0, 1, 2, \dots$}
        \State Set $(x^{s+1},\bar{f}^{s+1},\underline{f}^{s+1}) \gets \mathcal{GR}(\mu, x^{s}, \bar{f}^{s}, \underline{f}^{s}, m)$ with objective function $f$.
        \State \textbf{if} {$\bar{f}^{s+1} - \underline{f}^{s+1} \le \epsilon$} \Return $x^{s+1}$.
        \EndFor
    \end{algorithmic}
\end{algorithm}

Algorithms~\ref{alg:blm-fstarknown} and~\ref{alg:blm-fstarknown-approx} use the known optimal value $f^*$ to set the level. When $f^*$ is unavailable, we instead maintain upper and lower bounds on it and choose an intermediate level. This leads to the two-loop scheme in Algorithms~\ref{alg:Gap-Reduction} and~\ref{alg:blm-u}.

The outer loop, Algorithm~\ref{alg:blm-u}, keeps track of the best iterate and the current upper and lower bounds. In each iteration, it calls the inner Gap Reduction ($\GR$) subroutine to shrink their gap by a constant factor, until the gap is at most $\epsilon$.
Given these bounds, the inner loop, Algorithm~\ref{alg:Gap-Reduction}, runs Algorithm~\ref{alg:blm-fstarknown-approx} with the level fixed at $l=\underline{f}+\Delta_0/3$, where $\Delta_0=\bar{f}-\underline{f}$. When $l\geq f^*$, the level sets contain $X^*$, so distances to optimal solutions are non-increasing. The preceding convergence argument then guarantees sufficient objective decrease to trigger the upper-bound update in Line~\ref{GR:line:gap-reduction-term}.

The challenge is to update the lower bound using only the QG modulus $\mu$. Here we use a contrapositive argument: if $l\geq f^*$, the upper bound must improve after a fixed number of $T$ iterations. If this improvement does not occur, we can conclude that $l<f^*$ and raise the lower bound to $l$.
In the smooth setting, we have  $T=O(L/\mu)$. To avoid requiring $L$, we measure smoothness empirically from the generated iterates.

\begin{definition}[Empirical Smoothness Constant]
    \label{def:empirical-smoothness}
    For an approximately piecewise smooth (apx-PWS) objective $f$ and a gap estimate $\Delta\geq0$, the \emph{empirical smoothness constant} between distinct points $x,\bar{x}\in X$ is defined by\footnote{{For repeated points, set $\tilde L(x,x;\Delta)=+\infty$ if $f(x)-\tilde l_f(x;x)>\Delta/6$, and zero otherwise. We use $1/(+\infty)=0$ and $1/0=+\infty$.}}
    \begin{equation}
        \tilde{L}(x,\bar{x};\Delta) := \max\left\{ \frac{2\left(f(x) - \tilde{l}_f(x;\bar{x}) - \Delta/6\right)}{\|x-\bar{x}\|^2}, 0 \right\}.
        \label{eq:empirical-smoothness}
    \end{equation}
\end{definition}
For a $(k,L,\delta)$-apx-PWS function, $\tilde{L}(x,y;\Delta)\leq L$ whenever $x$ and $y$ lie in the same region and $\Delta\geq6\delta$.

In the PWS setting, we also need to account for matching pairs without knowing the smooth regions. The analysis in~\eqref{pfeq:del-accum-progress} suggests measuring progress along such a sequence by
\begin{equation}\label{eq:GR-empirical-progress}
    S(t;\{(x^{l_i},x^{r_i})\}):=\sum_{r_i\leq t}\frac{1}{\tilde{L}(x^{r_i},x^{l_i};\Delta_0)(r_i-l_i)}.
\end{equation}
The following lemma gives the threshold for the required upper-bound improvement; its proof is deferred to Subsection~\ref{subsec:blmu-analysis}.

\begin{lemma}[Empirical Progress Threshold]\label{lem:GR-progress-threshold}
    Let $f$ be convex and satisfy the QG condition with modulus $\mu>0$. Consider the iterates generated by the projection steps of Algorithm~\ref{alg:Gap-Reduction} with valid initial bounds, $\Delta_0>0$, and $l\geq\fstar$, without applying the termination criteria. If $S(T;\{(x^{l_i},x^{r_i})\})\geq6/\mu$ for an $m$-pair sequence and some $T\geq1$, then
    \[
        \min_{t\in[T]} f(x^t)\leq l+\frac{\Delta_0}{3}.
    \]
\end{lemma}
We use dynamic programming to find the pair sequence with the largest $S(t)$ (Line~\ref{GR:line:DP}). If this maximum reaches $6/\mu$ without the required upper-bound improvement, we raise the lower bound to $l$.

Before the precise convergence guarantee, we need to specify the pair sequences and statistics used to measure progress. Here, we relax the requirement that both points in each pair lie in the same smooth region: a small empirical smoothness constant between them suffices for fast convergence.

\begin{samepage}
    \begin{definition}[$m$-Pair Sequence and Empirical Smoothness Statistics]
        \label{def:empirical-matching-sequence}
        Given a $\GR$ call with initial gap $\Delta_0$ and iterates $\{x^j\}_{j=0}^N$, an \emph{$m$-pair sequence} is a sequence of iterate pairs $\{(x^{l_i},x^{r_i})\}_{i=1}^p$ whose indices satisfy $0\leq l_i<r_i\leq N$, $r_i\leq l_{i+1}$, and $r_i-l_i\leq m$. For $p\geq1$, we define the following statistics.\footnote{In the complexity bounds, the statistics are computed from the iterates generated before the final iteration of each call. Maxima are taken only over calls lasting more than $m$ iterations; if there are no such calls, the product $\bar{L}\taubar\bar{\sigma}$ is taken as zero.}
        \begin{itemize}
            \item \emph{Average inter-arrival time:} $\taubar:=N/p$.
            \item \emph{Average length:} $\bar{\sigma}:=p\left/\sum_{i=1}^p\frac{1}{r_i-l_i}\right.$.
            \item \emph{Average empirical smoothness:} $\bar{L}:=\left(\sum_{i=1}^p\frac{1}{r_i-l_i}\right)\left/\left(\sum_{i=1}^p\frac{1}{(r_i-l_i)\tilde{L}(x^{r_i},x^{l_i};\Delta_0)}\right)\right.$.
        \end{itemize}
    \end{definition}
\end{samepage}
We can now state the convergence guarantee for Algorithm~\ref{alg:blm-u}.

\begin{theorem}\label{thm:u-convergence}
    For problem~\eqref{eq:opt_prob}, assume $f$ is convex and satisfies the quadratic growth condition (Definition~\ref{def:quadartic_growth}) with modulus $\mu>0$. Consider Algorithm~\ref{alg:blm-u} with initial point $x^0\in X$, QG modulus $\mu$, $m\geq1$ cuts, and target accuracy $\epsilon>0$. For each $\GR$ call lasting more than $m$ iterations, choose an $m$-pair sequence with empirical statistics $(\bar{L}_s,\taubar_s,\bar{\sigma}_s)$ as in Definition~\ref{def:empirical-matching-sequence}, and let $\bar{L}$, $\taubar$, and $\bar{\sigma}$ be their respective maxima over these calls. The following bounds hold:

    a) The algorithm returns an $\epsilon$-optimal solution in $S$ outer iterations, where
    \[
        S\leq O(1)\left\lceil\log\left(\frac{\|f'(x^0)\|^2}{\mu\epsilon}\right)\right\rceil.
    \]
    The total number of oracle evaluations is bounded by
    \[
        O(1)\left(m+\left\lceil\frac{\bar{L}\taubar\bar{\sigma}}{\mu}\right\rceil\right)\left\lceil\log\left(\frac{\|f'(x^0)\|^2}{\mu\epsilon}\right)\right\rceil.
    \]
    For horizons of at least $m$ iterations, the pair sequences can be chosen with $\taubar\leq2m$ and $\bar{\sigma}\leq m$.

    b) Moreover, if $f$ is $(k,L,\delta)$-approximately piecewise smooth (Definition~\ref{def:apx-p.w.s}), $m\geq k$, and $\epsilon\geq6\delta$, then matching pair sequences can be chosen with $\bar{L}\leq L$.
\end{theorem}

Under the QG condition, Algorithm~\ref{alg:blm-u} inherits the main convergence guarantees of Algorithms~\ref{alg:blm-fstarknown} and~\ref{alg:blm-fstarknown-approx} without requiring the optimal value $f^*$. In particular, its oracle complexity depends on the empirical pair statistics, which reflect the local piecewise geometry near $X^*$.

In addition, the outer gap sequence $\{\Delta_s\}$ provides a computable certificate of suboptimality. This certificate remains valid even when $\epsilon<\delta$, although the linear convergence guarantee in Theorem~\ref{thm:u-convergence} requires $\epsilon\geq6\delta$. At higher accuracies, the empirical smoothness constant may increase and convergence may slow.

\subsection{The Convergence Analysis of the BL-$\mu$ Method}\label{subsec:blmu-analysis}

We first prove the progress threshold that justifies the lower-bound update in $\GR$.

\begin{proof}[Proof of Lemma~\ref{lem:GR-progress-threshold}]
    Suppose, to the contrary, that $S(T;\{(x^{l_i},x^{r_i})\})\geq6/\mu$ but $f(x^t)>l+\Delta_0/3$ for every $t\in[T]$. Since $l\geq f^*$, the affine minorants ensure that $X^*\subseteq X(j)$ at every iteration, so Lemma~\ref{lm:bridged-three-point-euclid} applies. For each pair $(x^{l_i},x^{r_i})$ completed by $T$, the assumed function-value bound gives
    \[
        f(x^{r_i})-l-\Delta_0/6>\Delta_0/6.
    \]
    For each such pair, the earlier cut is retained when $x^{r_i}$ is computed, giving $\tilde{l}_f(x^{r_i};x^{l_i})\leq l$. Hence $\tilde{L}(x^{r_i},x^{l_i};\Delta_0)>0$, and combining its definition with the bridged three-point inequality gives
    \[
        \frac{2[f(x^{r_i})-l-\Delta_0/6]}{\tilde{L}(x^{r_i},x^{l_i};\Delta_0)(r_i-l_i)}
        \leq\|x^{l_i}-x^*\|^2-\|x^{r_i}-x^*\|^2,
        \quad \forall x^*\in X^*.
    \]
    Summing over the pairs completed by $T$, using their non-overlapping intervals and distance monotonicity, and taking the infimum over $x^*\in X^*$ therefore yields
    \[
        \frac{\Delta_0}{3}S(T;\{(x^{l_i},x^{r_i})\})
        <\dist^2(x^0,X^*)
        \leq\frac{2[f(x^0)-f^*]}{\mu}
        \leq\frac{2\Delta_0}{\mu}.
    \]
    This contradicts $S(T;\{(x^{l_i},x^{r_i})\})\geq6/\mu$ and proves the claim. \qedsymbol
\end{proof}

The dynamic program in Line~\ref{GR:line:DP} computes $\tilde{S}_l(t)$ as the maximum progress over all $m$-pair sequences completed by $t$.

\begin{proof}[Proof of Theorem~\ref{thm:u-convergence}]
    \textit{Part a).} The initial lower bound is valid by Lemma~\ref{lem:initial_lower-bound}(a). In each $\GR$ call, an empty level set implies $l<f^*$, and Lemma~\ref{lem:GR-progress-threshold} gives the same conclusion if the empirical progress reaches $6/\mu$ before the upper-bound test succeeds. Thus, all returned lower bounds remain valid, while the best-iterate updates ensure $\bar{f}^s=f(x^s)$.

    Writing $\Delta_s=\bar{f}^s-\underline{f}^s$, each call lowers the upper bound by at least $\Delta_s/3$ or raises the lower bound by $\Delta_s/3$, so $\Delta_{s+1}\leq(2/3)\Delta_s$. Starting from $\Delta_0=2\|f'(x^0)\|^2/\mu$, this gives the stated bound on $S$. At termination, $f(x^S)-f^*\leq\Delta_S\leq\epsilon$.

    Consider the $\GR$ call at outer iteration $s$, assuming it lasts more than $m$ iterations, and let $N_s\geq m$ count the iterations before its final one. Since the accumulated progress remains below $6/\mu$ at those iterations, Definition~\ref{def:empirical-matching-sequence} gives
    \[
        \frac{N_s}{\taubar_s\bar{\sigma}_s\bar{L}_s}<\frac6\mu.
    \]
    Including the final iteration, the call uses at most
    \begin{equation}\label{eq:GR-iteration-bound}
        \bar{T}_s=m+1+\left\lceil\frac{6\taubar_s\bar{\sigma}_s\bar{L}_s}{\mu}\right\rceil
    \end{equation}
    iterations. The $m$ term also covers calls lasting at most $m$ iterations.

    Replacing the call-specific statistics in~\eqref{eq:GR-iteration-bound} by their respective maxima $(\bar{L},\taubar,\bar{\sigma})$ and summing over $S$ calls yields the total oracle bound. Choosing consecutive pairs $(x^{i-1},x^i)$, $i=1,\dots,N_s$, gives $\taubar_s=\bar{\sigma}_s=1$, which implies the stated bounds on $\taubar$ and $\bar{\sigma}$.

    \textit{Part b).} For calls made before the target gap is reached, $\Delta_s>\epsilon\geq6\delta$. Definition~\ref{def:empirical-smoothness} therefore gives $\tilde{L}(x^{r_i},x^{l_i};\Delta_s)\leq L$ for each same-region pair. Choosing a matching pair sequence for each call yields $\bar{L}_s\leq L$, and hence $\bar{L}\leq L$. The matching-pair construction also gives $\taubar_s\leq2k$ and $\bar{\sigma}_s\leq k$ when $m\geq k$. \qedsymbol
\end{proof}

\begin{remark}\label{rem:GR-guessed-parameters}
    For a positive initial gap, the iteration bound~\eqref{eq:GR-iteration-bound} and the algebraic gap reduction remain valid when the input $\mu>0$ or lower bound is misspecified; only the validity of the returned lower bound requires the correct assumptions.
\end{remark}

\subsection{The Weakly Convex Problem\protect\label{sec:WC}}

\begin{algorithm}
    \caption{The Inexact Proximal Point Method for $\rho$-Weakly Convex Problems}
    \label{alg:IPPM}
    \begin{algorithmic}[1]
        \Require Starting point $\bar{x}^{0}$, weak convexity modulus $\rho > 0$, and number of cuts $m$.

        \For{$s=0, 1, 2, \dots$}
        \State Set up the proximal subproblem: $P_{s}(x) := f(x) + \rho\|x-\bar{x}^{s}\|^2$; $\bar{P}^{0} \gets f(\bar{x}^{s})$; $\hat{x}^{0} \gets \bar{x}^{s}$; $\underline{P}^{0} \gets \bar{P}^{0} - \|f'(\bar{x}^{s})\|^2 / (2\rho)$.
        \State {Store $\bar\Delta_s\gets\bar P^0-\underline P^0$.}
        \For{$i=1, 2, 3, \dots$}
        \State $(\hat{x}^{i}, \bar{P}^{i}, \underline{P}^{i}) \leftarrow \GR(\rho, \hat{x}^{i-1}, \bar{P}^{i-1}, \underline{P}^{i-1}, m)$ with objective $P_s(x)$.
        \State {Update $\bar\Delta_s\gets\bar P^0-\underline P^i$.}
        \If{\label{IPPM:line:termination}$\bar{P}^{0}-P_{s}(\hat{x}^{i}) \ge P_{s}(\hat{x}^{i})-\underline{P}^{i}$}
        \State \label{IPPM:line:break} {Set $\bar{x}^{s+1} \gets \hat{x}^{i}$ and \textit{break}.}
        \EndIf
        \EndFor
        \EndFor
    \end{algorithmic}
\end{algorithm}

We now consider finding approximately stationary points for a weakly convex apx-PWS function, for which the weak convexity constant $\rho$ is known. We begin by reviewing the standard terminology for weakly convex optimization \cite{davis2019stochastic}.
\begin{definition}[Weak Convexity and Moreau Stationarity]
    \label{def:weak-convex}\label{def:M-stationarity}
    A function $f$ is \textit{$\rho$-weakly convex} if, for any $\bar{x} \in X$, the surrogate function $F_{2\rho}(x;\bar{x})$ is $\rho$-strongly convex with respect to $x$, where
    \begin{align}
        F_{2\rho}(x;\bar{x}) & :=f(x)+\rho\|x-\bar{x}\|^2, \label{eq:WC-perturbed-Moreau} \\
        f_{2\rho}(\bar{x})   & :=\min_{x\in X}F_{2\rho}(x;\bar{x}). \nonumber
    \end{align}
    Furthermore, a point $\bar{x}$ is \textit{$(\rho,\epsilon)$-Moreau stationary} for some $\epsilon \ge 0$ if $\|\rho(\bar{x}-\hat{x})\| \le \epsilon$, where $\hat{x} \leftarrow \argmin_{x\in X}F_{2\rho}(x;\bar{x})$.
\end{definition}

Since $\nabla f_{2\rho}(\bar{x})=2\rho(\bar{x}-\hat{x})$ \cite{davis2019stochastic}, our $(\rho,\epsilon)$-Moreau stationarity condition is equivalent to $\|\nabla f_{2\rho}(\bar{x})\|\leq2\epsilon$. The following lemma connects this stationarity measure to the function value gap of the surrogate problem. Its proof is deferred to the appendix.
\begin{lemma}
    \label{lem:relate-moreau-to-gap}
    For a $\rho$-weakly convex function $f$, let $F_{2\rho}$ and $f_{2\rho}$ be the perturbed function and its Moreau envelope defined in \eqref{eq:WC-perturbed-Moreau}. The gradient of the Moreau envelope is bounded by the function value gap as follows:
    \[
        \|\nabla f_{2\rho}(\bar{x})\|^2 \le 8\rho\left(F_{2\rho}(\bar{x};\bar{x})-\min_{x\in X}F_{2\rho}(x;\bar{x})\right) = 8\rho\left(f(\bar{x})-f_{2\rho}(\bar{x})\right).
    \]
\end{lemma}

Lemma~\ref{lem:relate-moreau-to-gap} bounds the Moreau stationarity measure at $\bar{x}$ in terms of the potential descent in the proximal objective $F_{2\rho}(x;\bar{x})$. This motivates the inexact proximal point method (IPPM) in Algorithm~\ref{alg:IPPM}. The method generates prox-centers $\{\bar{x}^s\}$ by approximately solving $F_{2\rho}(x;\bar{x}^s)$ to obtain each new prox-center $\bar{x}^{s+1}$. Specifically, at each outer iteration $s$, the method constructs the $\rho$-strongly convex proximal subproblem $P_s(x)$ and applies the \tGR subroutine (Algorithm~\ref{alg:Gap-Reduction}) until an iterate $\hat{x}^i$ achieves at least half of the maximal possible descent from $\bar{x}^s$. It then sets $\bar{x}^{s+1}=\hat{x}^i$ and proceeds to the next outer iteration.

    {The running gap $\bar\Delta_s$ certifies stationarity at $\bar x^s$ even before the descent test succeeds.} Using the original objective $f$ as a potential function for descent, we obtain the following convergence guarantee.

\begin{theorem}\label{thm:wc-complexity}
    For problem~\eqref{eq:opt_prob}, assume $f$ is $M$-Lipschitz continuous and $\rho$-weakly convex, with $\Delta_f:=f(\bar{x}^0)-\min_{x\in X}f(x)<\infty$. Consider the iterates $\{\bar{x}^s\}_{s\geq0}$ generated by Algorithm~\ref{alg:IPPM} with weak convexity modulus $\rho>0$, $m\geq1$ cuts, and initial point $\bar{x}^0\in X$. For each $\GR$ call lasting more than $m$ iterations, choose an $m$-pair sequence with empirical statistics $(\bar{L}_{s,i},\taubar_{s,i},\bar{\sigma}_{s,i})$ as in Definition~\ref{def:empirical-matching-sequence}, indexed by subproblem $s$ and call $i$, and let $\bar{L}$, $\taubar$, and $\bar{\sigma}$ be their respective maxima over these calls. {The following bounds hold when the empirical bounds are finite, including on nonterminal prefixes up to the target accuracy; part b) supplies sufficient conditions:}

    a) The algorithm generates a $(\rho,\epsilon)$-Moreau stationary iterate $\bar{x}^S$ {whose running gap satisfies $\bar{\Delta}_S\leq\epsilon^2/(8\rho)$ during outer iteration $S$}, where
    \[
        S\leq\ceil{16\rho\Delta_f/\epsilon^2}.
    \]
    The total number of oracle evaluations is bounded by
    \[
        {O(1)\left(1+\frac{\rho\Delta_f}{\epsilon^2}\right)\left(m+\ceil{\frac{\bar{L}\taubar\bar{\sigma}}{\rho}}\right)\ceil{\log\left(2+\frac{M}{\epsilon}\right)}.}
    \]
    The pair sequences can be chosen with $\taubar\leq2m$ and $\bar{\sigma}\leq m$.

    b) Moreover, if $f$ is $(k,L,\delta)$-approximately piecewise smooth, $m\geq k$, and {$\delta\leq\epsilon^2/(96\rho)$}, then matching pair sequences can be chosen with $\bar{L}\leq L+2\rho$.
\end{theorem}

We make two comments regarding the convergence result. First, it is useful to situate our complexity result in the context of more familiar settings.
\begin{itemize}
    \item For smooth, non-convex objectives (i.e., $(1,L)$-PWS functions), the oracle complexity bound simplifies to $O\left(\frac{\Delta_f L}{\epsilon^2}\log\frac{M}{\epsilon}\right)$, which nearly matches the well-known complexity of gradient descent \cite{LanBook,lan2023optimal}.

    \item For general non-smooth, $\rho$-weakly convex objectives, the Lipschitz conversion in~\eqref{eq:lipschitz-apx-pws} gives a $(1,2M^2/\delta+2\rho,\delta)$-apx-PWS model for each proximal objective $P_s$. Taking {$\delta=\epsilon^2/(96\rho)$} yields, for fixed $m$, the bound $O\left(\frac{\Delta_f\rho M^2}{\epsilon^4}\log\frac{M}{\epsilon}\right)$, matching the rate in~\cite{davis2019stochastic} up to a logarithmic factor.
\end{itemize}
A more refined analysis based on $\bar{\Delta}_s$, along the lines of~\cite{lan2023optimal}, allows us to remove the extra logarithmic factor in both the smooth and non-smooth cases.
The main advantage of Algorithm~\ref{alg:IPPM} arises for objectives that are apx-PWS. In this case, our method improves the general non-smooth rate from $\tilde{O}(1/\epsilon^4)$ to a much faster ``smooth-like'' rate of $\tilde{O}(1/\epsilon^2)$, with the additional factors depending on the local PWS geometry.

Second, the method requires only two parameters: the number of cuts $m$ and the weak convexity constant $\rho$. The algorithm is quite robust to the choice of $m$ (e.g., $m=10$ is often sufficient). It is, however, sensitive to the misspecification of $\rho$. We address this challenge in the subsequent sections by developing a verifiable criterion to detect whether $\rho$ is misspecified.

\subsection{The Convergence Analysis }

\begin{proof}[Proof of Theorem \ref{thm:wc-complexity}]
    {Let $\mathcal S$ be the first outer iteration in which the running gap satisfies $\bar\Delta_{\mathcal S}\leq\epsilon^2/(8\rho)$, possibly before the descent test succeeds. The lower bound is valid, so Lemma~\ref{lem:relate-moreau-to-gap} certifies Moreau stationarity at $\bar x^{\mathcal S}$. We bound the work needed to obtain this certificate.}

    \textit{Bounding the number of outer iterations.} We first calculate an upper bound on $S$ using a descent argument on $f$. {For any outer iteration $s<\mathcal S$, the stored gap when updating the prox-center satisfies $\bar\Delta_s>\epsilon^2/(8\rho)$.} The termination condition for the inner loop in Algorithm~\ref{alg:IPPM} is
    \[
        \bar{P}_s(\bar{x}^s) - P_s(\bar{x}^{s+1}) \ge P_s(\bar{x}^{s+1}) - \underline{P}_s^i,
    \]
    which implies $2[P_s(\bar{x}^s) - P_s(\bar{x}^{s+1})] \ge P_s(\bar{x}^s) - \underline{P}_s^i \ge \Delbar_s$. This gives the following descent guarantee for each outer iteration:
    \begin{align*}
        f(\bar{x}^s) - f(\bar{x}^{s+1}) & \ge f(\bar{x}^s) - \left[f(\bar{x}^{s+1}) + \rho\| \bar{x}^{s+1} - \bar{x}^s \|^2\right] \\
                                        & = P_s(\bar{x}^s) - P_s(\bar{x}^{s+1})                                                    \\
                                        & \ge \frac{1}{2}\Delbar_s.
    \end{align*}
    Summing this over all iterations from $s=0$ to $\mathcal{S}-1$ yields
    \begin{align*}
        \frac{\mathcal{S}\epsilon^2}{16\rho} \le \frac{1}{2}\sum_{s=0}^{\mathcal{S}-1}\Delbar_s \le \sum_{s=0}^{\mathcal{S}-1} \left(f(\bar{x}^s)-f(\bar{x}^{s+1})\right) = f(\bar{x}^0) - f(\bar{x}^{\mathcal{S}}) \le f(\bar{x}^0)-\min_{x\in X}f(x) = \Delta_f.
    \end{align*}
    Rearranging gives the bound on the number of outer loops, $\mathcal{S} \le \ceil{16\rho\Delta_f/\epsilon^2}$.

    \textit{Bounding the total complexity.} Next, we bound the number of oracle evaluations within each outer loop $s$. Since $f$ is $M$-Lipschitz, the initial gap of the subproblem is bounded: $\Delta_s^0 \le M^2/(2\rho)$. The subproblem $P_s(x)$ satisfies the QG condition with modulus $\rho$. {Before the running gap reaches the target, failure of the descent test gives $\bar P^i-\underline P^i>\bar\Delta_s/2>\epsilon^2/(16\rho)$. The residual gap contracts by at least a factor of $2/3$ per $\GR$ call, so the work before either the descent test succeeds or the target running gap is reached is bounded by}
    \[
        {O(1)\left(m+\ceil{\frac{\bar{L}\taubar\bar{\sigma}}{\rho}}\right)\ceil{\log\left(2+\frac{M}{\epsilon}\right)}.}
    \]
    {Summing over the $\mathcal S+1$ visited prox-centers gives the total oracle complexity.}

    \textit{The apx-PWS case.} Finally, if $f$ is a $(k,L,\delta)$-apx-PWS function, then the proximal function $P_s(x)$ is a $(k, L+2\rho, \delta)$-apx-PWS function. {Every $\GR$ input gap before the target certificate is obtained exceeds $\epsilon^2/(16\rho)\geq6\delta$. Thus the empirical-smoothness argument in the proof of Theorem~\ref{thm:u-convergence}(b) gives $\bar{L}\leq L+2\rho$.}\qedsymbol
\end{proof}

\section{Verifiable Termination Condition }\label{sec:termination}

The $\mu$-BL method (Algorithm~\ref{alg:blm-u}) and IPPM (Algorithm~\ref{alg:IPPM}) require knowledge of a lower-curvature parameter, which is often unavailable in practice. We develop a normalized Wolfe-gap stationarity (\tcW-stationarity) certificate for PWS functions to verify algorithmic progress without this knowledge.

We first identify the properties needed for a termination criterion and define the \tcW-stationarity certificate. For convex functions, we present an algorithm to compute the certificate and show that, under QG, its implied optimality bound is proportional to the input gap estimate. We then extend the construction to weakly convex functions through convex proximal subproblems and relate the resulting certificate to Moreau stationarity. The next section uses these guarantees to develop almost parameter-free algorithms for both settings.

\subsection{Requirements for a Termination Criterion}

Before presenting our \tcW-stationarity certificate, it is useful to recall why the gradient norm, $\|\nabla f(\bar{x})\|$, is a widely accepted termination certificate for the simpler setting of smooth and strongly convex optimization \cite{lan2023optimal}. Two properties are essential. First, it is easily \emph{verifiable}: given a point $\bar{x}$, one can compute $\|\nabla f(\bar{x})\|$ without knowledge of any other problem parameters. Second, the gradient norm provides an accurate characterization of the optimality gap:
\begin{equation}\label{eq:grad-norm}
    \frac{1}{2L}\|\nabla f(\bar{x})\|^2 \le f(\bar{x})-f^* \le \frac{1}{2\mu}\|\nabla f(\bar{x})\|^2.
\end{equation}
Given an evaluation center $\xbar$, the right-hand inequality shows that the gradient norm \emph{upper bounds} the optimality gap, while the left-hand inequality shows that the gradient norm evaluated at $\xbar$ is \emph{proportional} to the gap. These two properties are critical for developing optimal parameter-free algorithms for smooth optimization, as in \cite{lan2023optimal}.

These observations motivate the following requirements for a termination certificate in the PWS setting.

\begin{definition}[Requirements for a Reasonable Certificate]
    \label{req:termination-condition}
    We call a termination certificate $V(\bar{x})$ for an apx-PWS function \textit{reasonable} if it satisfies the following conditions:
    \begin{description}[align=left, leftmargin=*, font=\normalfont\itshape]
        \item[\emph{Handles Constraints}:] The certificate can be applied to problems with a simple, closed, convex feasible region $X$.
              \vgap
        \item[\emph{Verifiable}:] The certificate can be verified by calling the black-box first-order oracle without requiring any problem parameters (e.g., $L$ or $\mu$).
              \vgap
        \item[\emph{Provides an Upper Bound}:] Under the QG setting, the certificate provides a computable upper bound on the optimality gap, $f(\bar{x}) - f^*$.
              \vgap
        \item[\emph{Computable \& Proportional}:] For any feasible point $\bar{x}$, there exists a finite-time, first-order method to generate a certificate whose value is proportional to the optimality gap.
    \end{description}
\end{definition}

\begin{table}
    \centering

    \begin{tabular}{lccccc}
        \toprule
                                      &            & \multicolumn{3}{c}{Optimality Gap} &                                                 \\
        \cmidrule(lr){3-5}
        Certificate                   & Verifiable & Upper Bound                        & Proportional & Computable & Handles Constraints \\
        \midrule
        Optimality Gap                & $\times$   & \checkmark                         & \checkmark   & $\times$   & \checkmark          \\
        Gradient Mapping              & \checkmark & \checkmark                         & $\times$     & \checkmark & \checkmark          \\
        Moreau Stationarity           & \checkmark & \checkmark                         & \checkmark   & $\times$   & \checkmark          \\
        Approx. Moreau                & $\times$   & \checkmark                         & \checkmark   & \checkmark & \checkmark          \\
        Goldstein Stationarity        & \checkmark & \checkmark                         & ?            & \checkmark & $\times$            \\
        \midrule
        \tcW-Stationarity (This work) & \checkmark & \checkmark                         & \checkmark   & \checkmark & \checkmark          \\
        \bottomrule
    \end{tabular}
    \caption{Comparison of termination certificates for PWS problems against the criteria from Definition~\ref{req:termination-condition}.}\label{tab:different-termination-criterion}
\end{table}

As summarized in Table~\ref{tab:different-termination-criterion}, common stationarity certificates do not satisfy all these requirements, motivating our \tcW-stationarity certificate.
\begin{itemize}
    \item The \textit{optimality gap}, $f(x^t)-f^*$, provides an ideal characterization of progress. However, it is not a practical certificate because $f^*$ cannot be computed or verified by any first-order method.

    \item The \textit{gradient norm} (or the gradient mapping for constrained problems) is computationally attractive, requiring only a single gradient evaluation. However, it fails to be proportional to the optimality gap in the PWS setting due to non-smoothness. For example, with $f(x):=|x|$, we have $\|\nabla f(x)\|=1$ for all $x>0$, while the gap $f(x)-f^* = x$ approaches zero. %

    \item The \textit{Moreau stationarity} certificate $M_{\eta}(\bar{x}):=\eta(\bar{x}-\hat{x}_{\eta})$, where $\hat{x}_{\eta}\leftarrow\argmin_{x\in X}f(x)+\frac{\eta}{2}\|x-\bar{x}\|^2$, is proportional to the optimality gap but is not computable in finite time because it requires the exact minimizer $\hat{x}_{\eta}$ \cite{davis2019stochastic}. Approximating $\hat{x}_{\eta}$ requires yet another termination certificate for the proximal subproblem.

    \item The \textit{Goldstein stationarity} certificate is another popular choice for non-smooth problems \cite{zhang2020complexity}. Its primary drawbacks are its inability to handle simple constraints specified by $X$ and its reliance on a carefully chosen radius parameter $\delta(\bar{x})$ to maintain proportionality with the function value gap.
\end{itemize}

\subsection{The W-Stationarity Certificate}

Inspired by the Wolfe-gap termination condition and the Goldstein stationarity certificate \cite{zhang2020complexity}, we propose the following \tcW-stationarity certificate to meet the requirements in Definition~\ref{req:termination-condition}.

\begin{definition}[\tcW-Stationarity Certificate]
    \label{def:W-stationary}
    Given an apx-PWS function $f$, an evaluation center $\bar{x}\in X$, a radius $\iota>0$, and a set of search points $\{x^i\}_{i\in[m]}\subset B(\bar{x};\iota)$, we first define the model function $\psi(x)$ as the maximum of affine approximations to $f$:
    \[
        \psi(x) := \max\left\{ \max_{i\in[m]}\{\tilde{l}_f(x;x^i)\}, \; l_f(x;\bar{x}) \right\},
    \]
    where $l_f(x;\bar{x}) := f(\bar{x}) + \langle f'(\bar{x}), x-\bar{x} \rangle$ . The normalized Wolfe-gap (or \textit{\tcW-gap}) is then defined as the maximal rate of descent on this model within a local ball:
    \begin{equation}
        \mathcal{V}_{\iota}(\bar{x};\{x^i\}_{i\in[m]}) := \frac{1}{\iota} \left\{ \max_{x \in X} \left( \psi(\bar{x})-\psi(x) \right) \quad \text{s.t.} \quad \|x-\bar{x}\| \le \iota \right\}. \label{eq:Vio-def}
    \end{equation}
    We say the set $\{x^i\}_{i\in[m]}$ constitutes an \textit{$(\iota,\nu)$-\tcW-stationarity certificate} for $\bar{x}$ if $\|x^i-\bar{x}\| \le \iota$ for all $i\in[m]$, and $\mathcal{V}_{\iota}(\bar{x}) \le \nu$. Moreover, if the search points are clear from the context, we would drop
    $\{\xt i\}_{i\in m}$ to write $\Vio(\xbar)$ for simplicity.
\end{definition}

A few remarks are in order. First, the \tcW-gap measures the steepness of the model $\psi$ in the neighborhood $B(\bar{x};\iota)$ of the evaluation center $\bar{x}$. As $\iota \to 0_+$, the \tcW-gap reduces to the gradient norm $\|\nabla f(\bar{x})\|$. Second, our proposed certificate differs from the classic Wolfe-gap in three important ways:
\begin{itemize}
    \item It uses a model function $\psi(x)$ constructed from a \textit{bundle} of affine approximations evaluated at multiple search points. This construction is important for handling the non-smoothness of PWS functions.
    \item It focuses on descent within a local $\iota$-ball, yielding a less conservative stationarity measure than one based on descent over the entire feasible region $X$.
    \item It restricts the search points $\{x^i\}$ to lie within the same $\iota$-ball.
\end{itemize}
Moreover, for any fixed set of search points, the \tcW-gap is computed by solving a quadratically constrained quadratic program (QCQP), making the certificate \textit{verifiable} without knowledge of hidden problem parameters. Additionally, the certificate \textit{handles constraints} because its definition incorporates the feasible region $\mathcal{X}$. The following subsections address the remaining requirements in Definition~\ref{req:termination-condition}: constructing the certificate and establishing its optimality guarantees under QG.

\subsection{Computing the \tcW-stationarity Certificate}

We focus on convex objectives in this subsection and the next. Here, we propose an algorithm to find a \tcW-stationarity certificate satisfying
\begin{equation}\label{eq:W-gap-lower-bound}
    \frac{\mathcal{V}_{\iota}(\bar{x})^2}{\tilde{L}} \lesssim f(\bar{x})-f^*,
\end{equation}
where $\tilde{L}$ is an empirical smoothness constant for the apx-PWS function $f$, and the radius $\iota$ is suitably chosen. This mirrors the lower bound in the gradient norm relation~\eqref{eq:grad-norm}.

\begin{algorithm}
    \caption{The \tcW-Certificate Search Subroutine}
    \label{alg:W-certificate-search}
    \begin{algorithmic}[1]
        \Require Convex objective $f$; evaluation center $\bar{x}\in X$; estimated optimality gap $\Delta>0$; number of search points $m\ge1$; maximal radius $\iota_{\max}\in(0,+\infty]$.
        \Ensure An $(\iota,\nu)$-\tcW-stationarity certificate or False; a certificate is guaranteed if $f(\bar{x})-f^*\le\Delta$.

        \State \textbf{Initialize:} $x^0 \gets \bar{x}$, $t \gets 0$, $\tilde{L}\gets0$, $l \gets f(\bar{x}) - 2\Delta$, and $\tilde{l}_f(x;x^0)\gets l_f(x;x^0) \gets f(x^0) + \langle f'(x^0), x - x^0 \rangle$.

        \For{$t=0, 1, \dots, m$}
        \parState{\label{W-cert:line:projection} Compute the projection:
            $$x^{t+1} \gets \argmin_{x \in X(t)} \|x - \bar{x}\|^2 \text{ where } X(t) := \{x \in X \mid l_f(x;x^0) \le l \text{ and } \tilde{l}_f(x;x^i) \le l, \forall i \in \{1,\dots,t\}\}$$}

        \State \textbf{if} the computation is infeasible, set $\norm{\xt{t+1}-\xbar}=+\infty$.
        \If{\label{W-cert:line:iomax}$\norm{\xt{t+1}-\xbar}>\iotamax$} \Comment{\textit{A finite $\iota_{\max}$ only for non-convex setting.}}
        \State \Return $\{\xt i\}_{i\in[t]}$ as the $(\iotamax,\frac{2\Delta}{\iotamax})$ \tcW-stationarity certificate.
        \EndIf
        \State \label{W-cert:line:infeasibility}\textbf{if }$\norm{\xt{t+1}-\xbar}=+\infty,$ \Return $\{\xt i\}_{i\in[t]}$ as the $(\bigM,\frac{2\Delta}{\bigM})$ \tcW-stationarity
        certificate.

        \EndFor

        \State \label{W-cert:line:empirical-sm}Compute empirical smoothness constant: $\tilde{L} \gets \min_{1 \le l < r \le m+1} \frac{2\left[f(x^r) - \tilde{l}_f(x^r;x^l) - \Delta/6\right]_+}{\|x^l - x^r\|^2}$.
        \State \label{W-cert:line:fail}\textbf{if} $\tilde{L}=0$ or $\|\bar{x} - x^{m+1}\| < \frac{1}{2}\sqrt{\frac{\Delta}{\tilde{L}}}$ \Return False.
        \State \label{W-cert:line:last}Set radius: $\iota \gets \|\bar{x} - x^{m+1}\|$, \Return search points $\{x^i\}_{i \in [m]}$ as an $(\iota, 2\Delta/\iota)$ \tcW-stationarity certificate.
    \end{algorithmic}
\end{algorithm}

As shown in \textbf{Algorithm \ref{alg:W-certificate-search}}, the subroutine takes as input the evaluation center $\bar{x}$, an estimated optimality gap $\Delta$, the number of search points $m$, and a maximal search radius $\iota_{\max}$. Its main step follows the level-set projection in Algorithm~\ref{alg:blm-fstarknown}, with projections taken from the fixed evaluation center $\bar{x}$. Crucially, the level-set parameter is set to $l = f(\bar{x}) - 2\Delta$, a value intentionally chosen to be below the optimal value $f^*$. This encourages the search to explore different smooth regions of the function. The subroutine terminates under one of several conditions:
\begin{itemize}
    \item If the search distance $\|x^{t+1} - \bar{x}\|$ exceeds the maximal radius $\iota_{\max}$, the algorithm returns a certificate based on this radius (Line~\ref{W-cert:line:iomax}).
    \item If the projection becomes infeasible, the objective function is bounded below by $l$. The algorithm returns a certificate (Line~\ref{W-cert:line:infeasibility}). For analytical purposes, we use a large constant $\bigM$ as a proxy for an infinite radius, and the certificate conditions continue to hold as $\bigM \to \infty$.
    \item If the subroutine completes $m+1$ iterations without early termination, it computes the empirical smoothness $\tilde{L}$ from the generated search points.\footnote{Here, $[a]_+:=\max\{a,0\}$. In the definition of $\tilde{L}$, a quotient with zero denominator is interpreted as $+\infty$ if its numerator is positive and as $0$ otherwise. We use $1/(+\infty)=0$ and evaluate the square-root test only when $\tilde{L}>0$.} It then returns `False' if $\tilde{L}=0$ or the final search point is too close to the evaluation center (Line~\ref{W-cert:line:fail}); otherwise, it returns a \tcW-certificate based on the final distance (Line~\ref{W-cert:line:last}). When $\Delta$ upper bounds the optimality gap $f(\bar{x})-f^*$, the subroutine returns the desired \tcW-certificate, as established in the following lemma.
\end{itemize}

\begin{lemma}
    \label{lem:W-certificate-validity}
    Suppose $f$ is convex and Algorithm~\ref{alg:W-certificate-search} is run with inputs $\bar{x}\in X$, $\Delta>0$, $m\ge1$, and $\iota_{\max}\in(0,+\infty]$. The following properties hold:
    \begin{enumerate}
        \item[a)] The distances from the search points to the evaluation center, $\{\|x^i-\bar{x}\|\}$, are monotonically non-decreasing.

        \item[b)] The search points returned from Lines~\ref{W-cert:line:iomax}, \ref{W-cert:line:infeasibility}, or~\ref{W-cert:line:last} form an $(\iota,\nu)$-\tcW-stationarity certificate, with
              \[
                  \mathcal{V}_\iota(\bar{x})\le\nu=\frac{2\Delta}{\iota}.
              \]
              If the certificate is returned from Line~\ref{W-cert:line:infeasibility} or~\ref{W-cert:line:last}, then $\nu\le4\sqrt{\Delta\tilde{L}}$.

        \item[c)] If $f$ is $(k,L,\delta)$-apx-PWS with $k\le m$ and $\delta\le\Delta/6$, then $\tilde{L}\le L$ whenever the empirical smoothness is computed.

        \item[d)] If $f(\bar{x})-f^*\le\Delta$, the algorithm does not return `False'.
    \end{enumerate}
\end{lemma}

\begin{proof}
    Let $\psi^0(x):=l_f(x;\bar{x})$ and, for $t\ge1$, let
    \[
        \psi^t(x):=\max\left\{l_f(x;\bar{x}),\max_{1\le i\le t}\tilde{l}_f(x;x^i)\right\}.
    \]
    Since the affine models are minorants of $f$, we have $\psi^t(\bar{x})=f(\bar{x})$.

    \textit{Part a):}
    The level sets $X(t)=\{x\in X:\psi^t(x)\le l\}$ are nested: $X(t)\subseteq X(t-1)$. Thus, whenever both projections are feasible, the distances from $\bar{x}$ to these sets satisfy $\|x^{t+1}-\bar{x}\|\ge\|x^t-\bar{x}\|$. The first distance is at least $\|x^0-\bar{x}\|=0$. If a projection is infeasible, the algorithm assigns its distance the value $+\infty$, so monotonicity still holds.

    \textit{Part b):}
    Suppose the returned certificate uses the model $\psi^{\bar{t}}$, where $\bar{t}=m$ for a return from Line~\ref{W-cert:line:last}. All returned search points lie in $B(\bar{x};\iota)$ by part a) and the radius checks; in the infeasible case, choose $\bigM$ large enough to contain them.

    For each return, the projection rule implies that $\psi^{\bar{t}}(x)>l$ whenever $x\in X$ and $\|x-\bar{x}\|<\iota$: the level set is either empty or has distance at least $\iota$ from $\bar{x}$. By continuity, approaching a boundary point along the segment from $\bar{x}$ gives $\psi^{\bar{t}}(x)\ge l$ throughout $X\cap B(\bar{x};\iota)$. Consequently,
    \[
        \mathcal{V}_\iota(\bar{x})
        \le\frac{f(\bar{x})-l}{\iota}
        =\frac{2\Delta}{\iota}=\nu.
    \]
    For a return from Line~\ref{W-cert:line:last}, the preceding test gives $\iota\ge\frac12\sqrt{\Delta/\tilde{L}}$, and hence $\nu\le4\sqrt{\Delta\tilde{L}}$. For a return from Line~\ref{W-cert:line:infeasibility}, the same bound holds as $\bigM\to+\infty$.

    \textit{Part c):}
    If the empirical smoothness is computed, the pigeonhole principle gives indices $1\le i<j\le m+1$ whose search points lie in the same smooth region. By approximate piecewise smoothness,
    \[
        f(x^j)-\tilde{l}_f(x^j;x^i)
        \le\frac{L}{2}\|x^j-x^i\|^2+\delta.
    \]
    If $x^i\ne x^j$, then $\delta\le\Delta/6$ and the definition in Line~\ref{W-cert:line:empirical-sm} imply
    \[
        \tilde{L}
        =\min_{1\le l<r\le m+1}
        \frac{2\bigl[f(x^r)-\tilde{l}_f(x^r;x^l)-\Delta/6\bigr]_+}{\|x^l-x^r\|^2}
        \le\frac{2\bigl[f(x^j)-\tilde{l}_f(x^j;x^i)-\delta\bigr]_+}{\|x^j-x^i\|^2}
        \le L.
    \]
    If $x^i=x^j$, the corresponding quotient is interpreted as zero, so $\tilde{L}=0\le L$.

    \textit{Part d):}
    Suppose $f(\bar{x})-f^*\le\Delta$ and the algorithm reaches Line~\ref{W-cert:line:empirical-sm}. For any $1\le i<j\le m+1$, the constraint generated at $x^i$ is included in $X(j-1)$, so $\tilde{l}_f(x^j;x^i)\le l$. Therefore,
    \[
        f(x^j)-\tilde{l}_f(x^j;x^i)-\frac{\Delta}{6}
        \ge f^*-l-\frac{\Delta}{6}
        \ge\frac{5\Delta}{6}>0.
    \]
    Thus the positive part leaves these numerators unchanged, and $\tilde{L}>0$. If $\tilde{L}=+\infty$, the algorithm cannot return `False'. Otherwise, suppose it does, and choose a pair $(x^i,x^j)$ attaining the minimum that defines $\tilde{L}$. Part a) and the triangle inequality give
    \[
        \|x^i-x^j\|
        \le\|x^i-\bar{x}\|+\|x^j-\bar{x}\|
        \le2\|x^{m+1}-\bar{x}\|
        <\sqrt{\frac{\Delta}{\tilde{L}}}.
    \]
    Using $\tilde{l}_f(x^j;x^i)\le l$ and the definition of this pair, we obtain
    \begin{align*}
        f(x^j)
         & =\tilde{l}_f(x^j;x^i)+\frac{\tilde{L}}{2}\|x^i-x^j\|^2+\frac{\Delta}{6} \\
         & <l+\frac{\Delta}{2}+\frac{\Delta}{6}
        =f(\bar{x})-\frac{4\Delta}{3}
        <f^*,
    \end{align*}
    contradicting the definition of $f^*$. \qedsymbol
\end{proof}

For convex problems, we take $\iota_{\max}=+\infty$, so Lemma~\ref{lem:W-certificate-validity}(b) ensures that every returned certificate satisfies the desired relation~\eqref{eq:W-gap-lower-bound} when the gap estimate $\Delta\ge f(\bar{x})-f^*$.\footnote{Here, $\Delta$ should not be too large relative to the actual gap, i.e., $\Delta\lesssim f(\bar{x})-f^*$.}

\subsection{Proportional \tcW-Stationarity Certificate under the QG Setting}

We now establish the complementary relation under the QG condition for a suitably chosen radius $\iota$:\footnote{It suffices that $\iota\le2\mathcal{V}_\iota(\bar{x})/\mu$.}
\[
    f(\bar{x})-f^* \lesssim \frac{\mathcal{V}_{\iota}(\bar{x})^2}{\mu}.
\]
We first establish this bound for a given certificate, then combine it with the preceding construction to obtain a proportional optimality guarantee.

Showing that the \tcW-certificate provides a meaningful \textit{upper bound} on the optimality gap requires a preliminary monotonicity result. The following lemma shows that for a fixed set of search points, the \tcW-gap is a monotonically non-increasing function of its radius $\iota$.

\begin{lemma}
    \label{lem:Vdelta-monotonicity-1}
    Fixing the search points $\{x^i\}_{i\in[m]}$, the \tcW-gap $\mathcal{V}_\iota(\bar{x})$ is a monotonically non-increasing function of its radius $\iota$ for any $\iota>0$ satisfying $\iota \ge \max_{i\in[m]}\|x^i-\bar{x}\|$.
\end{lemma}
\begin{proof}
    Let $\iota_1\ge\iota_2>0$ be admissible radii, and let $x_{\iota_1}$ be an optimal solution in the definition of the \tcW-gap~\eqref{eq:Vio-def}. By the convexity of $X$, the point $\tilde{x}:=\bar{x}+\frac{\iota_2}{\iota_1}(x_{\iota_1}-\bar{x})$ belongs to $X$ and satisfies $\|\tilde{x}-\bar{x}\|\le\iota_2$. Convexity of $\psi$ gives
    \[
        \psi(\tilde{x})\le\left(1-\frac{\iota_2}{\iota_1}\right)\psi(\bar{x})
        +\frac{\iota_2}{\iota_1}\psi(x_{\iota_1}).
    \]
    Therefore, by the definition of the \tcW-gap,
    \[
        \mathcal{V}_{\iota_2}(\bar{x})
        \ge\frac{\psi(\bar{x})-\psi(\tilde{x})}{\iota_2}
        \ge\frac{\psi(\bar{x})-\psi(x_{\iota_1})}{\iota_1}
        =\mathcal{V}_{\iota_1}(\bar{x}).
    \]
    \qedsymbol
\end{proof}

The next lemma shows that the \tcW-certificate provides an upper bound on the optimality gap for a convex function satisfying the QG condition. %

\begin{lemma}
    \label{lem:W-certificate-str-meaningful}
    Given a convex apx-PWS objective function $f$, suppose there exists an $(\iota,\nu)$-\tcW-stationarity certificate for a feasible point $\bar{x}$. If $\bar{x}$ satisfies the QG condition $f(\bar{x})-f^*\ge\frac{\mu}{2}\text{dist}^2(\bar{x},X^*)$ for some $\mu > 0$, then the \tcW-stationarity certificate implies the following error bound condition:
    \[
        \dist(\xbar,X^*)\le\max\left\{\iota,\frac{2\nu}{\mu}\right\}.
    \]
    Moreover, its optimality gap is bounded by:
    \begin{equation}
        f(\bar{x})-f^* \le \max\left\{\iota\nu, \frac{2\nu^2}{\mu}\right\}. \label{eq:W-certificate-meaningful}
    \end{equation}
\end{lemma}
\begin{proof}
    By definition, $\psi(\bar{x})=f(\bar{x})$. We bound the distance and optimality gap by considering two cases based on the location of the projection $x_p^* := \text{proj}_{X^*}(\bar{x})$.

    \textit{Case 1: $\| \bar{x}-x_p^* \| \le \iota$.} In this case, $x_p^*$ is inside the ball used to define the \tcW-gap. Since $f$ is convex, $f(x_p^*) \ge \psi(x_p^*)$. We have:
    \[
        f(\bar{x})-f^* \le \psi(\bar{x})-\psi(x_p^*) \le \max_{x\in B(\bar{x};\iota)\cap X} \left( \psi(\bar{x})-\psi(x) \right) = \iota \mathcal{V}_{\iota}(\bar{x}) \le \iota\nu.
    \]

    \textit{Case 2: $\| \bar{x}-x_p^* \| > \iota$.} We first bound the optimality gap by the distance. Using the monotonicity of the \tcW-gap from Lemma~\ref{lem:Vdelta-monotonicity-1}, we get:
    \[
        f(\bar{x})-f^* \le \psi(\bar{x})-\psi(x_p^*) \le \| \bar{x}-x_p^* \| \mathcal{V}_{\|\bar{x}-x_p^*\|}(\bar{x}) \le \| \bar{x}-x_p^* \| \mathcal{V}_{\iota}(\bar{x}) \le \nu\|\bar{x}-x_p^*\|.
    \]
    Now, we incorporate the QG condition with the inequality we just derived:
    \[
        \frac{\mu}{2}\|\bar{x}-x_p^*\|^2 \le f(\bar{x})-f^* \le \nu\|\bar{x}-x_p^*\|.
    \]
    Solving this for the distance gives $\|\bar{x}-x_p^*\| \le \frac{2\nu}{\mu}$. Substituting this back into our gap inequality yields:
    \[
        f(\bar{x})-f^* \le \nu\|\bar{x}-x_p^*\| \le \nu\left(\frac{2\nu}{\mu}\right) = \frac{2\nu^2}{\mu}.
    \]
    Combining the two cases gives both bounds. \qedsymbol
\end{proof}

Combining the optimality bound in Lemma~\ref{lem:W-certificate-str-meaningful} with the search guarantees in Lemma~\ref{lem:W-certificate-validity}, we now show that the computed certificate provides an optimality bound proportional to the input gap estimate $\Delta$.

\begin{proposition}
    \label{prop:proportional-W-ceritificate}
    Given a convex and $(k,L,\delta)$-apx-PWS objective function $f$, consider running Algorithm~\ref{alg:W-certificate-search} with inputs $\bar{x}$, $\Delta$, $m$, and $\iota_{\max}=+\infty$. The output satisfies the following properties.
    \begin{enumerate}
        \item[a)] If $\Delta$ is a valid upper bound on the optimality gap (i.e., $f(\bar{x})-f^* \le \Delta$), then the algorithm is guaranteed to return a \tcW-certificate from either Line~\ref{W-cert:line:infeasibility} or Line~\ref{W-cert:line:last}.

        \item[b)] If the algorithm terminates at Line~\ref{W-cert:line:infeasibility}, the returned \tcW-certificate implies an optimality gap bound of $2\Delta$.

        \item[c)] If the algorithm terminates at Line~\ref{W-cert:line:last} and the objective satisfies the $\mu$-QG condition, the returned certificate implies an optimality gap of at most $\max\{2, \frac{32\tilde{L}}{\mu}\}\Delta$, where $\tilde{L}$ is the empirical smoothness constant from Line~\ref{W-cert:line:empirical-sm}. If, in addition, $\delta \le \frac{1}{6}\Delta$ and $m \ge k$, this bound tightens to $\max\{2, \frac{32L}{\mu}\}\Delta$.
    \end{enumerate}
\end{proposition}
\begin{proof}
    \textit{Part a)}: Since $\iota_{\max}=+\infty$, the termination condition in Line~\ref{W-cert:line:iomax} is never met. As the input $\Delta$ is a valid upper bound on the optimality gap, Lemma~\ref{lem:W-certificate-validity}.d) ensures that the algorithm will not return `False` from Line~\ref{W-cert:line:fail}. Therefore, the algorithm must terminate by returning a certificate from either Line~\ref{W-cert:line:infeasibility} or Line~\ref{W-cert:line:last}.

    \textit{Part b)}: A return from Line~\ref{W-cert:line:infeasibility} means that the projection subproblem is infeasible. Since all cuts are minorants of $f$, any $x\in X$ satisfying $f(x)\le l$ would be feasible. Hence
    \[
        f^*\ge l=f(\bar{x})-2\Delta
        \quad\Longrightarrow\quad
        f(\bar{x})-f^*\le2\Delta.
    \]

    \textit{Part c)}: If the certificate is returned from Line~\ref{W-cert:line:last}, its parameters are $\iota = \|x^{m+1}-\bar{x}\|$ and $\nu=2\Delta/\iota$. The condition for termination in this line is $\iota \ge \frac{1}{2}\sqrt{\Delta/\tilde{L}}$, which implies $\nu \le 4\sqrt{\Delta\tilde{L}}$. Applying the bound from Lemma~\ref{lem:W-certificate-str-meaningful} yields:
    \[
        f(\bar{x})-f^* \le \max\left\{\iota\nu, \frac{2\nu^2}{\mu}\right\} \le \max\left\{2\Delta, \frac{2(4\sqrt{\Delta\tilde{L}})^2}{\mu}\right\} = \max\left\{2, \frac{32\tilde{L}}{\mu}\right\}\Delta.
    \]
    Furthermore, if $\Delta \ge 6\delta$ and $m \ge k$, Lemma~\ref{lem:W-certificate-validity}.c) guarantees that $\tilde{L} \le L$, which gives the implied upper bound. \qedsymbol
\end{proof}

A few remarks are in order. First, Proposition~\ref{prop:proportional-W-ceritificate} shows that the optimality gap bound provided by the \tcW-stationarity certificate is proportional to the gap estimate $\Delta$, up to a factor of order $\max\{1,L/\mu\}$. This satisfies the final requirement for a ``reasonable'' certificate in Definition~\ref{req:termination-condition}. Second, a key practical advantage is that, to obtain this bound, the number of cuts $m$ need only exceed the number of smooth regions in the local neighborhood $B(\bar{x};\iota)$, rather than the total number globally. Finally, the result highlights the importance of a bundle of cuts. If one sets $m=1$ and treats a general $M$-Lipschitz continuous function as a $(1, 2M^2/\Delta, \Delta)$-apx-PWS function, the empirical condition number $\tilde{L}/\mu$ can become arbitrarily large. In this case, the certificate value $\nu=O(\sqrt{\Delta\tilde{L}})$ can degrade to $O(M)$ even when the point $\bar{x}$ is near optimal. This is expected, as a single-cut model is equivalent to the subgradient norm, which cannot provide a proportional optimality guarantee for general non-smooth functions.

\subsection{Extension to Non-convex Problems}

{For a $\rho$-weakly convex objective $f$, we seek to certify Moreau stationarity using the \tcW-certificate, which can be evaluated without knowing $\rho$. Here, $\tilde{l}_f(\cdot;x^i)$ denotes an affine approximation of $f$ at the search point $x^i$.\footnote{We may use exact subgradient cuts $l_f(x;x^i):=f(x^i)+\langle f'(x^i),x-x^i\rangle$, or inexact affine approximations satisfying $\tilde{l}_f(x;x^i)\le f(x)+\rho\|x-x^i\|^2$ for all $x\in X$. This condition also ensures that the surrogate cuts in~\eqref{eq:W-surrogate-cuts} are affine minorants.}  The next lemma shows that an $(\iota,\nu)$-\tcW-certificate implies that its evaluation center is $O(\nu)$-Moreau stationary when $\iota$ is chosen appropriately relative to $\nu$ and $\rho$.}

\begin{lemma}
    \label{lem:W-certificate-to-M-stationarity}
    Given a $\rho$-weakly convex function $f$ with $\rho>0$ and affine approximations as specified above, if there exists an $(\iota,\nu)$-\tcW-stationarity certificate for a point $\bar{x}$, then $\bar{x}$ is a $(\rho, 2\nu+4\iota\rho)$-Moreau stationary point.
\end{lemma}
\begin{proof}
    Let $\hat{x}:=\argmin_{x\in X}F_{2\rho}(x;\bar{x})$. If $\|\hat{x}-\bar{x}\|\le\iota$, the conclusion follows immediately. Otherwise, consider the point
    \[
        z:=\bar{x}+\frac{\iota}{\|\hat{x}-\bar{x}\|}(\hat{x}-\bar{x})\in X\cap B(\bar{x};\iota).
    \]
    Let $\psi_f$ be the certificate's model. The cut condition and $\|z-x^i\|\le2\iota$ imply $\psi_f(z)\le f(z)+4\rho\iota^2$. Since the center cut gives $\psi_f(\bar{x})\ge f(\bar{x})$, the certificate bound yields
    \begin{align*}
        F_{2\rho}(\bar{x};\bar{x})-F_{2\rho}(z;\bar{x})
         & =f(\bar{x})-f(z)-\rho\iota^2              \\
         & \le\psi_f(\bar{x})-\psi_f(z)+3\rho\iota^2
        \le\iota\nu+3\rho\iota^2.
    \end{align*}
    On the other hand, $\rho$-strong convexity along the segment from $\bar{x}$ to $\hat{x}$, together with $F_{2\rho}(\bar{x};\bar{x})-F_{2\rho}(\hat{x};\bar{x})\ge\frac{\rho}{2}\|\hat{x}-\bar{x}\|^2$, gives
    \[
        F_{2\rho}(\bar{x};\bar{x})-F_{2\rho}(z;\bar{x})
        \ge\rho\iota\|\hat{x}-\bar{x}\|-\frac{\rho}{2}\iota^2.
    \]
    Combining these bounds yields
    \[
        \rho\|\hat{x}-\bar{x}\|\le\nu+\frac{7}{2}\rho\iota
        \le2\nu+4\rho\iota,
    \]
    which proves the stated Moreau stationarity bound. \qedsymbol
\end{proof}

\begin{remark}
    For unconstrained problems, an $(\iota,\nu)$-\tcW-stationarity certificate also implies $(\iota+\bar{\delta},\nu)$-Goldstein stationarity when the model gradients are evaluated within $\bar{\delta}$ of the search points; see Lemma~\ref{lem:W-certificate-to-G-stationarity} in the appendix.
\end{remark}

{For use in the IPPM framework (Algorithm~\ref{alg:IPPM}), we next connect \tcW-certificates for the convex proximal surrogate $P_{\bar{x}}(x):=F_{2\rho}(x;\bar{x})=f(x)+\rho\|x-\bar{x}\|^2$ to certificates for $f$. We form the surrogate cuts by combining the given affine approximation of $f$ with the exact linearization of the quadratic regularizer:}
\begin{equation}\label{eq:W-surrogate-cuts}
    {\tilde{l}_{P_{\bar{x}}}(x;x^i)
    :=\tilde{l}_f(x;x^i)+\rho\|x^i-\bar{x}\|^2
    +2\rho\langle x^i-\bar{x},x-x^i\rangle.}
\end{equation}
{The next proposition specifies the parameters for Algorithm~\ref{alg:W-certificate-search} and establishes the resulting certificate and Moreau-stationarity guarantees.}

\begin{proposition}
    \label{prop:pr-WC-W-certificate}
    Consider a $\rho$-weakly convex function $f$ with $\rho>0$ and a point $\bar{x}$ satisfying $F_{2\rho}(\bar{x};\bar{x})-\min_{x\in X}F_{2\rho}(x;\bar{x}) \le \Delta_{\bar{x}} = \epsilon^2/\rho$ for some $\epsilon>0$. When Algorithm~\ref{alg:W-certificate-search} is applied to the surrogate function $F_{2\rho}(\cdot;\bar{x})$ with inputs $(\bar{x}, \Delta_{\bar{x}}, m, \iota_{\max}=\sqrt{\Delta_{\bar{x}}/\rho})$ and the paired cuts in~\eqref{eq:W-surrogate-cuts}, the returned $(\iota,\nu)$-\tcW-certificate has the following properties.
    \begin{enumerate}
        \item[a)] The certificate parameter satisfies $\nu \le \max\{2,4\sqrt{\tilde{L}/\rho}\}\epsilon$, where $\tilde{L}$ is the empirical smoothness constant from Line~\ref{W-cert:line:empirical-sm}. The same search points also constitute a $(\iota, 2\nu)$-\tcW-stationarity certificate for the original objective function $f$, and $\bar{x}$ is a $(\rho,6\nu)$-Moreau stationary point.
        \item[b)] If the surrogate function $F_{2\rho}(\cdot;\bar{x})$ is a $(k,L,\delta)$-apx-PWS function with $m \ge k$ and $\delta \le \Delta_{\bar{x}}/6$, then the empirical smoothness is bounded by $\tilde{L} \le L$.
    \end{enumerate}
\end{proposition}
\begin{proof}
    {The cuts in~\eqref{eq:W-surrogate-cuts} are affine minorants of $P_{\bar{x}}$.} Let $\psi_{P_{\bar{x}}}$ and $\psi_f$ denote the corresponding models. Their center cuts agree, so $\psi_{P_{\bar{x}}}(\bar{x})=f(\bar{x})$.

    \textit{Part a)}: Since $\Delta_{\bar{x}}$ is an upper bound on the optimality gap of the convex function $P_{\bar{x}}$, Lemma~\ref{lem:W-certificate-validity} ensures that Algorithm~\ref{alg:W-certificate-search} returns a \tcW-certificate. If $\iota=\iota_{\max}$, then $\nu=2\Delta_{\bar{x}}/\iota_{\max}=2\epsilon$. Otherwise, the final return gives $\iota\ge\frac12\sqrt{\Delta_{\bar{x}}/\tilde{L}}$, and hence $\nu\le4\sqrt{\tilde{L}\Delta_{\bar{x}}}=4\sqrt{\tilde{L}/\rho}\epsilon$. These two cases give the stated bound on $\nu$.

    {To transfer the certificate to $f$, rewrite~\eqref{eq:W-surrogate-cuts} as
    \[
        \tilde{l}_{P_{\bar{x}}}(x;x^i)
        =\tilde{l}_f(x;x^i)+\rho\|x-\bar{x}\|^2-\rho\|x-x^i\|^2.
    \]
    Together with $\|x^i-\bar{x}\|\le\iota$, this gives, for every $x\in X\cap B(\bar{x};\iota)$,}
    \[
        \psi_f(\bar{x})\le f(\bar{x})+\rho\iota^2,
        \qquad
        \psi_f(x)\ge\psi_{P_{\bar{x}}}(x)-\rho\iota^2.
    \]
    Since $\iota\le\iota_{\max}=\sqrt{\Delta_{\bar{x}}/\rho}$ and $\nu=2\Delta_{\bar{x}}/\iota$, we also have $\nu\ge2\iota\rho$. Therefore,
    \begin{align*}
         & \max_{x\in X\cap B(\bar{x};\iota)}\bigl(\psi_f(\bar{x})-\psi_f(x)\bigr)                                               \\
         & \qquad\le\max_{x\in X\cap B(\bar{x};\iota)}\bigl(\psi_{P_{\bar{x}}}(\bar{x})-\psi_{P_{\bar{x}}}(x)\bigr)+2\rho\iota^2 \\
         & \qquad\le\iota\nu+2\rho\iota^2\le2\iota\nu.
    \end{align*}
    Thus, the same search points form an $(\iota,2\nu)$-\tcW-stationarity certificate for $f$. Lemma~\ref{lem:W-certificate-to-M-stationarity} then gives $(\rho,6\nu)$-Moreau stationarity, since $2(2\nu)+4\iota\rho\le6\nu$.

    \textit{Part b)}: If the empirical smoothness is computed, Lemma~\ref{lem:W-certificate-validity}(c) gives $\tilde{L}\le L$. Otherwise, $\tilde{L}$ retains its initial value of zero, so the bound still holds. \qedsymbol
\end{proof}

Two remarks are in order. First, this result shows that our generated \tcW-certificate is proportional to the Moreau stationarity measure, up to a factor related to the condition number $\sqrt{L/\rho}$. Second, the guarantee is reliable up to the inexactness parameter $\delta$; if the target gap $\Delta_{\bar{x}}$ is smaller than $6\delta$, the empirical smoothness constant $\tilde{L}$ may no longer be bounded by $L$.

\section{Almost Parameter-Free Bundle Level Method}\label{sec:pf}

In this section, we use the proposed \tcW-stationary certificate to design almost parameter-free algorithms for convex QG problems and weakly convex problems. Since the \tcW-certificate is both verifiable and proportional (see Definition~\ref{req:termination-condition}), we use it to track if an algorithm is making the desired progress and adjust the algorithm accordingly in an adaptive fashion. This allows our methods to take advantage of unknown growth conditions in the apx-PWS  setting.

\subsection{Convex Problems with an Unknown QG Constant}

\begin{algorithm}
    \caption{Almost Parameter-Free BL method (pfBL-$\mu$) for Convex QG Objectives}
    \label{alg:pf-BL-u}
    \begin{algorithmic}[1]
        \Require An initial point $x^0\in X$; an upper bound on the QG modulus, $\tilde{\mu}$; the number of cuts $m$.

        \State Initialize: $\bar{f}^0 \gets f(x^0)$, $\underline{f}^0 \gets f(x^0) - \frac{2\|f'(x^0)\|^2}{\tilde{\mu}}$, $\tau \gets 0$, and $\Delta_0 \gets \bar{f}^0 - \underline{f}^0$.

        \For{$s=0, 1, 2, \dots$}
        \State \label{pfBL:line:GR} Run $(x^{s+1}, \bar{f}^{s+1}, \underline{f}^{s+1}) \leftarrow \GR(\tilde{\mu}, x^s, \bar{f}^s, \underline{f}^s, m)$ with objective $f$ and  set $\Delta_{s+1} \gets \bar{f}^{s+1} - \underline{f}^{s+1}$.

        \If{\label{pfBL:line:half-flag}$\Delta_{s+1} \le \frac{1}{2}\Delta_{\tau}$}
        \State Set $\tau \gets s+1$.
        \State \label{pfBL:line:W-cert}Run Algorithm~\ref{alg:W-certificate-search} on $f$ with inputs $(x^{\tau}, \Delta_{\tau}, m, +\infty)$ to generate a \tcW-stationarity certificate.
        \State \label{pfBL:line:W-cert-fail} \textbf{if} Algorithm~\ref{alg:W-certificate-search} returns \textit{False}, restart Algorithm~\ref{alg:pf-BL-u} with inputs $(x^{s+1}, \tilde{\mu}/2, m)$.
        \EndIf
        \EndFor
    \end{algorithmic}
\end{algorithm}

We first consider the convex setting where the QG constant $\mu$ is unknown and the goal is to minimize the function value gap. {Algorithm~\ref{alg:pf-BL-u} requires} only an initial point $x^0$, the number of cuts $m$, and an initial upper bound {$\tilde{\mu}$ on $\mu$}. As discussed in Lemma~\ref{lem:initial_lower-bound}(b), this upper bound can be computed efficiently.

More concretely, Algorithm~\ref{alg:pf-BL-u} employs a "guess-and-check" strategy. It {uses} the $\mu$-BL method (Algorithm~\ref{alg:blm-u}) with a guessed QG constant $\tilde{\mu}$ to shrink the optimality gap $\Delta_s$. Once the gap halves (Line~\ref{pfBL:line:half-flag}), it runs the \tcW-certificate search algorithm (Line~\ref{pfBL:line:W-cert}) to verify this progress. {If the search returns \textit{False}, indicating that the guess $\tilde{\mu}$ was too large, the algorithm restarts with a halved guess, $\tilde{\mu}/2$ (Line~\ref{pfBL:line:W-cert-fail}). The following theorem bounds the number of oracle evaluations required to find an $\epsilon$-optimal solution.}

\begin{theorem}
    \label{thm:pf_QG-complexity}
    Consider an $M$-Lipschitz continuous convex objective function $f$ that satisfies the QG condition with modulus $\mu>0$. {Run Algorithm~\ref{alg:pf-BL-u} with inputs $(x^0, \tilde{\mu}, m)$, where $\tilde{\mu}\geq\mu$ and $m\geq1$, and fix $\epsilon>0$.
            The number of restarts, $Q$, is bounded by $Q\leq\ceil{\log_2(\tilde\mu/\mu)}$.} Moreover, the algorithm's efficiency is characterized as follows:
    \begin{enumerate}
        \item[a)] {Let $\bar{L},\taubar,\bar\sigma$ be finite upper bounds on the empirical statistics of Definition~\ref{def:empirical-matching-sequence} throughout each $\GR$ call lasting more than $m$ iterations, and let $\hat L$ be a finite upper bound on the empirical smoothness constants observed during calls} to Algorithm~\ref{alg:W-certificate-search}\footnote{If Line~\ref{W-cert:line:empirical-sm} is never invoked, $\hat{L}$ is regarded as zero.} before an $\epsilon$-optimal solution is generated. {The algorithm then generates an $\epsilon$-optimal solution, and the number of required oracle evaluations is bounded by:}
              \[
                  {O(1)\left(\ceil{\frac{\bar{L}\taubar\bar{\sigma}}{\mu}}+m\left(1+\ceil{\log\frac{\tilde{\mu}}{\mu}}\right)\right)\ceil{\log\left(2+\frac{M^2\max\{\hat{L},\tilde{\mu}\}}{\mu^2\epsilon}\right)}.}
              \]

        \item[b)] {If $f$ is also a $(k,L,\delta)$-apx-PWS function with $k \le m$ and $\delta \le \frac{\epsilon\mu}{1200\max\{L,\tilde\mu\}}$, then the finite bounds required in part a) hold for some matching pair sequences with $\bar{L}\leq L$, $\hat{L}\leq L$, $\taubar\leq2k$, and $\bar\sigma\leq k$.}
    \end{enumerate}
\end{theorem}
\begin{proof}
    For clarity in the analysis, we index quantities by the restart counter $q$. Thus, $\tilde{\mu}_q$ is the guessed QG modulus and $x^{q,s}$ is the $s$-th iterate during the $q$-th restart phase.

        {First, set $\bar\Delta:=\mu\epsilon/(32\max\{\hat L,\tilde\mu\})$. If a certificate search is called with $\Delta_{q,s}\leq\bar\Delta$, it either returns \textit{False}, triggering a restart, or Proposition~\ref{prop:proportional-W-ceritificate} implies that its evaluation center is $\epsilon$-optimal. Each $\GR$ call contracts the estimated gap by at least a factor of $2/3$, and a certificate is checked at least once every two calls. Thus, each phase reaches an $\epsilon$-optimal point or restarts after at most
            \[
                2+\left\lceil\frac{\log\left(1+2M^2/(\tilde\mu_q\bar\Delta)\right)}{\log(3/2)}\right\rceil
            \]
            $\GR$ calls, since $\Delta_{q,0}=2\|f'(x^{q,0})\|^2/\tilde\mu_q\leq2M^2/\tilde\mu_q$. The finite empirical bounds and~\eqref{eq:GR-iteration-bound} ensure that each call finishes. Including certificate searches, the cost per call is at most
            \[
                \left\lceil\frac{6\taubar\bar\sigma\bar L}{\tilde\mu_q}\right\rceil+2m+2.
            \]

            Once $\tilde\mu_q\leq\mu$, the lower bounds are valid and Proposition~\ref{prop:proportional-W-ceritificate}(a) rules out further restarts. Hence $Q\leq\lceil\log_2(\tilde\mu/\mu)\rceil$, and every phase has $\tilde\mu_q>\mu/2$. Summing the geometric terms $1/\tilde\mu_q$ and the per-phase $m$ terms gives part a).

            For part b), set $\bar\Delta:=\mu\epsilon/(32\max\{L,\tilde\mu\})$. Before an $\epsilon$-optimal point is generated, each phase starts with $\Delta_{q,0}>\mu\epsilon/\tilde\mu_q\geq32\bar\Delta$ by Lemma~\ref{lem:initial_lower-bound}(a). Each $\GR$ call satisfies $\Delta_{q,s+1}\geq\Delta_{q,s}/3$: nontrivial projections have an active minorant at level $l$, while lower-bound updates follow a failed upper-bound test. Until the first certificate check with gap at most $\bar\Delta$ or an earlier restart, the halving rule therefore keeps every $\GR$ input gap above $\bar\Delta/2$ and every checked gap above $\bar\Delta/6\geq6\delta$. The matching-pair argument of Theorem~\ref{thm:u-convergence}(b) and Lemma~\ref{lem:W-certificate-validity}(c) give the claimed empirical bounds, ensuring finite $\GR$ calls. The phase then reaches an $\epsilon$-optimal point or restarts as in part a). \qedsymbol
        }\end{proof}

{A few remarks are in order. First, for $(k,L,\delta)$-apx-PWS functions, Algorithm~\ref{alg:pf-BL-u} matches the $O(\frac{Lk^2}{\mu}\log(\frac{1}{\epsilon}))$  complexity bounds of Algorithms~\ref{alg:blm-u} and~\ref{alg:blm-fstarknown-approx}, without requiring $\mu$, $f^*$, or $\epsilon$ as inputs. Second, for exact PWS cuts with $m\geq k$, the rate holds at every accuracy. With fixed inexact cuts, the guarantee requires $\epsilon\geq1200\delta\max\{L,\tilde\mu\}/\mu$. Third, the argument also extends to bounded feasible regions without global QG. If $X$ has diameter $D_X$, taking $\mu=\epsilon/D_X^2$ before an $\epsilon$-optimal point is generated leads to an oracle complexity of $O(\frac{Lk^2D_X^2}{\epsilon}\log(\frac{1}{\epsilon}))$.}

\subsection{Nonconvex Problems with an Unknown Weak Convexity Constant}

\begin{algorithm}
    \caption{An Almost Parameter-Free IPPM for $\rho$-Weakly Convex Problems}
    \label{alg:pf-IPPM}
    \begin{algorithmic}[1]
        \Require Initial point $\bar{x}^0\in X$; initial guess $\tilde\rho>1$; number of cuts $m\geq1$; stationarity requirement $\epsilon>0$.
        \Ensure A point and an $(\iota,\nu)$-\tcW-certificate with $\iota,\nu\leq\epsilon$.
        \For{$s=0,1,2,\dots$}
        \parState{Set $P_s(x):=f(x)+\tilde\rho\|x-\bar x^s\|^2$, $\bar P^0\gets f(\bar x^s)$, $\hat x^0\gets\bar x^s$, and $\underline P^0\gets\bar P^0-\|f'(\bar x^s)\|^2/(2\tilde\rho)$.}
        \State \textbf{if} $f'(\bar x^s)=0$, \Return $\bar x^s$ with its singleton $(\epsilon,0)$-\tcW-certificate.
        \For{$i=0,1,2,\dots$}
        \State Set $\Delta_s\gets\bar P^0-\underline P^i$.
        \State \label{pfPM:line:W-cert} Run Algorithm~\ref{alg:W-certificate-search} on $P_s$ with inputs $(\bar x^s,\Delta_s,m,\sqrt{\Delta_s/\tilde\rho})$.
        \parState{If a certificate $(\iota_s,\nu_s)$ is returned, evaluate the corresponding $(\iota_s,\nu_s^+)$-\tcW-certificate for $f$ on the same search points using Definition~\ref{def:W-stationary}.}
        \If{\label{pfPM:line:restart-trigger}the search returns \textit{False} or $\nu_s^+>2\nu_s$}
        \State \Return Algorithm~\ref{alg:pf-IPPM} with inputs $(\hat x^i,2\tilde\rho,m,\epsilon)$.
        \ElsIf{$\nu_s^+\leq\epsilon$ and $\iota_s\leq\epsilon$}
        \State \label{pfPM:line:return} \Return $\bar x^s$ and its $(\iota_s,\nu_s^+)$-\tcW-certificate.
        \EndIf
        \If{\label{pfPM:line:opt-halve-trigger}$\bar P^0-P_s(\hat x^i)\geq P_s(\hat x^i)-\underline P^i$}
        \State Set $\bar x^{s+1}\gets\hat x^i$ and \textit{break}.
        \EndIf
        \State $(\hat x^{i+1},\bar P^{i+1},\underline P^{i+1})\leftarrow\GR(\tilde\rho,\hat x^i,\bar P^i,\underline P^i,m)$ on $P_s$.
        \EndFor
        \EndFor
    \end{algorithmic}
\end{algorithm}

We now consider the non-convex apx-PWS setting where the weak convexity modulus $\rho$ is unknown. Using the IPPM method (Algorithm~\ref{alg:IPPM}) with a misspecified constant $\tilde{\rho}$ poses two challenges. First, if $\tilde{\rho} < \rho$, the proximal subproblems may not be convex, rendering the lower bounds from the \tGR subroutine unreliable. {Second, when $\tilde{\rho}$ is misspecified, a \tcW-stationarity certificate for the surrogate need not yield a comparable certificate for the original objective $f$. Consequently, moving to a new prox-center may not yield sufficient descent relative to the stationarity of $f$.}

    {We address these issues through a ``guess-and-check'' strategy using \tcW-certificates, which can be computed without knowing $\rho$.} Algorithm~\ref{alg:pf-IPPM} takes an initial point $x^0$, an initial guess $\tilde{\rho}$ for the weak convexity constant, the number of cuts $m$, and a stationarity requirement $\epsilon$. It runs the IPPM method with the current guess $\tilde{\rho}$, constructing the surrogate cuts via~\eqref{eq:W-surrogate-cuts}. {Before each $\GR$ call, it generates a surrogate certificate at the current prox-center and evaluates a certificate for $f$ on the same search points (Line~\ref{pfPM:line:W-cert}). If the search fails (indicating an invalid lower bound for the surrogate function) or $\nu_s^+>2\nu_s$ (the \tcW-stationarity of the surrogate function is no longer comparable to that of the original objective), the algorithm restarts from the best current point with a doubled guess, $2\tilde\rho$.} The efficiency of this adaptive method is established in the next theorem.

\begin{theorem}
    \label{thm:pf-WC-complexity}
    Consider an $M$-Lipschitz continuous, $\rho$-weakly convex objective function $f$ that is bounded from below with $f(\xt 0)-\min_{x\in X}f(x)\leq \Delta_f<+\infty$. Run Algorithm~\ref{alg:pf-IPPM} with inputs $(x^0, \tilde{\rho}_0, m, \epsilon)$, where $\tilde{\rho}_0 > 1$ is an initial guess for the weak convexity modulus, and set {$\rho_{\max} := \max\{\tilde{\rho}_0,2\rho\}$}. The following statements hold:
    \begin{enumerate}
        \item[a)] {Let $\bar L,\taubar,\bar\sigma$ be finite upper bounds on the empirical statistics of Definition~\ref{def:empirical-matching-sequence} throughout each $\GR$ call lasting more than $m$ iterations, and let $\hat L$ be a finite upper bound on the empirical smoothness from calls} to Algorithm~\ref{alg:W-certificate-search}.\footnote{If Line~\ref{W-cert:line:empirical-sm} is never invoked, $\hat{L}$ is taken to be zero.} {The algorithm then returns an $(\iota,\nu)$-\tcW-stationary point $\bar{x}$ with $\iota,\nu\leq\epsilon$, and the total number of oracle evaluations is bounded by:}
              \[
                  {\begin{aligned}
                               & O(1)\left(1+\frac{\Delta_f(\hat{L}+\rho_{\max})}{\epsilon^2}\right)
                              \left(\ceil{\frac{\bar{L}\taubar\bar{\sigma}}{\tilde{\rho}_0}}+m\left(1+\ceil{\log\frac{\rho_{\max}}{\tilde{\rho}_0}}\right)\right) \\
                               & \qquad\times\ceil{\log\left(2+\frac{M(\hat{L}+\rho_{\max})}{\epsilon}\right)},
                          \end{aligned}}
              \]
              where the $m$-pair statistics satisfy $\taubar \le 2m$ and $\bar{\sigma} \le m$.

        \item[b)] {If $f$ is also a $(k,L,\delta)$-apx-PWS function with $k \le m$ and $\delta \le \frac{\epsilon^2}{6000(\rho_{\max}+L)}$, then matching pair sequences give the finite bounds required in part a), with $\bar{L} \le L+2\rho_{\max}$ and $\hat{L} \le L+2\rho_{\max}$.}
    \end{enumerate}
\end{theorem}

\begin{proof}
    {We index restart phases by $q$, with guess $\tilde\rho_q$. If $\tilde\rho_q\geq\rho$, the surrogate is $\tilde\rho_q$-strongly convex and its cuts are minorants. The $\GR$ lower bounds are then valid, so $\Delta_s=\bar P^0-\underline P^i$ upper bounds the surrogate gap at $\bar x^s$. Proposition~\ref{prop:pr-WC-W-certificate} rules out both restart conditions. Thus the number of phases is $O(1+\log(\rho_{\max}/\tilde\rho_0))$, and every guess is at most $\rho_{\max}$.

        Set $\Delta_\epsilon:=\epsilon^2/[144(\rho_{\max}+\hat L)]$. At any check that does not trigger a restart, the search formulas and the acceptance test give
        \[
            \iota_s\leq\sqrt{\Delta_s/\tilde\rho_q},
            \qquad \nu_s^+\leq2\nu_s\leq8\sqrt{\max\{\hat L,\rho_{\max}\}\Delta_s}.
        \]
        These inequalities use the radius-cap return and the final empirical-smoothness test directly, so they also hold when $\tilde\rho_q<\rho$. Since $\tilde\rho_q>1$, a check with $\Delta_s\leq\Delta_\epsilon$ must terminate or restart. The zero-subgradient return already gives the claimed certificate.

        If the algorithm instead moves to a new prox-center, the descent test implies
        \[
            f(\bar x^s)-f(\bar x^{s+1})
            \geq\bar P^0-P_s(\hat x^i)\geq\Delta_s/2>\Delta_\epsilon/2.
        \]
        Restarting from $\hat x^i$ also cannot increase $f$, because $P_s(\hat x^i)\leq\bar P^0$. Each phase therefore visits at most $1+2\Delta_f/\Delta_\epsilon$ prox-centers. If neither the certificate nor the descent test ends an inner loop, its residual gap satisfies
        \[
            \bar P^i-\underline P^i>\tfrac12(\bar P^0-\underline P^i)
            =\Delta_s/2>\Delta_\epsilon/2.
        \]
        This residual contracts by a factor of at least $2/3$ per $\GR$ call, starting from at most $M^2/(2\tilde\rho_q)$. The finite empirical bounds and~\eqref{eq:GR-iteration-bound} therefore give a per-phase cost of
        \[
            O(1)\left(1+\frac{\Delta_f}{\Delta_\epsilon}\right)
            \left(m+\left\lceil\frac{\bar L\taubar\bar\sigma}{\tilde\rho_q}\right\rceil\right)
            \left\lceil\log\left(2+\frac{M^2}{\tilde\rho_q\Delta_\epsilon}\right)\right\rceil.
        \]
        This includes the $O(m)$ cost of each certificate search. Summing the geometric terms $1/\tilde\rho_q$ and the per-phase $m$ terms proves part a).

        For part b), set $\Delta_\epsilon:=\epsilon^2/[144(L+3\rho_{\max})]$. The initial search at each prox-center uses only the exact center cut and either terminates or has input gap greater than $\Delta_\epsilon$. As in the proof of Theorem~\ref{thm:pf_QG-complexity}(b), consider the calls up to the first certificate check with gap at most $\Delta_\epsilon$ or an earlier restart. Each $\GR$ input gap exceeds $\Delta_\epsilon/2\geq6\delta$, and a $\GR$ update leaves the next search gap at least two thirds of the preceding one. Thus every subsequent checked gap exceeds $2\Delta_\epsilon/3\geq6\delta$. The matching-pair argument of Theorem~\ref{thm:u-convergence}(b) and Lemma~\ref{lem:W-certificate-validity}(c), applied to the surrogate, give $\bar L,\hat L\leq L+2\rho_{\max}$ and ensure finite $\GR$ calls. A check with gap at most $\Delta_\epsilon$ then terminates or restarts as in part a). \qedsymbol
    }\end{proof}

We make two remarks regarding this result. {First, compared with Algorithm~\ref{alg:IPPM}, which assumes that $\rho$ is known, Algorithm~\ref{alg:pf-IPPM} retains the same $\tilde{O}(1/\epsilon^2)$ dependence on accuracy, with an extra factor of order $\bar L/\tilde{\rho}_0$ when $\hat L+\rho_{\max}=O(\bar L)$. This factor arises from checking the \tcW-stationarity certificate, which requires solving the surrogate subproblems to higher accuracy. A practical rule is therefore to choose the initial guess $\tilde{\rho}_0$ on the order of the smoothness scale $L$ to avoid a large penalty. Second, Lemma~\ref{lem:W-certificate-to-M-stationarity} converts the resulting certificates into Moreau stationarity guarantees of the same order.}

\bibliographystyle{siam}
\bibliography{Reference}

\section{Appendix }

\begin{lemma}\label{lem:approx-smooth-transfer}
    Let the function $f: \mathcal{X} \to \mathbb{R}$ be $M$-Lipschitz continuous. Suppose that for two points $\tilde{x}, \tilde{y} \in \mathcal{X}$, the smoothness condition $f(\tilde{x}) -l_f(\tilde{x};\tilde{y})\leq \frac{L}{2}\|\tilde{x} -\tilde{y}\|^2$ holds for some $L \ge 0$. Then for any points $x\in B(\tilde{x};\bar{\delta})$ and $y\in B(\tilde{y};\bar{\delta})$ in the $\bar{\delta}$-neighborhoods of $\tilde{x}$ and $\tilde{y}$ respectively, we have:
    \begin{equation*}
        f(x)-l_f(x;\tilde{y}) \leq L\|x-y\|^2 + 2M\bar{\delta} + 4L\bar{\delta}^2.
    \end{equation*}
\end{lemma}
\begin{proof}
    We decompose the left-hand side into three parts:
    \[
        f(x)-l_f(x;\tilde{y}) = \underbrace{(f(x)-f(\tilde{x}))}_{\text{Term 1}} + \underbrace{(f(\tilde{x}) - l_f(\tilde{x};\tilde{y}))}_{\text{Term 2}} + \underbrace{(l_f(\tilde{x};\tilde{y}) -l_f(x;\tilde{y}))}_{\text{Term 3}}.
    \]
    We bound each of these terms separately.
    \begin{itemize}
        \item \textit{Term 1:} By the $M$-Lipschitz continuity of $f$ and since $x \in B(\tilde{x}; \bar{\delta})$, we have:
              \[ f(x)-f(\tilde{x})\leq M\|x-\tilde{x}\|\leq M\bar{\delta}. \]

        \item \textit{Term 2:} Using the lemma's assumption and standard norm inequalities, we bound this term:
              \begin{align*}
                  f(\tilde{x}) - l_f(\tilde{x};\tilde{y})
                   & \leq \frac{L}{2}\|\tilde{x} -\tilde{y}\|^2                                                                \\
                   & \le \frac{L}{2}\left(2\|x-y\|^2 + 2\|(\tilde{x}-x) + (y-\tilde{y})\|^2 \right)                            \\
                   & \leq L\|x-y\|^2 + L\left(\|\tilde{x}-x\| + \|y-\tilde{y}\|\right)^2 \quad (\text{by triangle inequality}) \\
                   & \leq L\|x-y\|^2 + L(2\bar{\delta})^2 = L\|x-y\|^2 + 4L\bar{\delta}^2.
              \end{align*}

        \item \textit{Term 3:} By the definition of $l_f(\cdot;\tilde{y})$ and the fact that $\|f'(\tilde{y})\|_* \le M$:
              \[
                  l_f(\tilde{x};\tilde{y}) - l_f(x;\tilde{y}) = \langle f'(\tilde{y}), \tilde{x}-x \rangle \leq \|f'(\tilde{y})\|_* \|\tilde{x}-x\| \leq M\bar{\delta}.
              \]
    \end{itemize}
    Summing the bounds for the three terms gives the desired result. \qedsymbol
\end{proof}

\begin{proof}[Proof to Lemma \ref{lem:relate-moreau-to-gap}]

    Let $\xhat\leftarrow\argmin_{x\in X}F_{2\ups}(x;\xbar)$. It follows
    from Lemma 2.2 in \cite{davis2019stochastic} that $\grad f_{2\ups}(\xbar)=2\ups(\bar{x}-\xhat)$.
    Since $F_{2\ups}$ is strongly convex with modulus $\ups$, we have
    \begin{align*}
        \normsq{\grad f_{2\ups}(\xbar)} & =(8\ups)\frac{\ups}{2}\normsq{\xbar-\xhat}           \\
                                        & \leq8\ups[F_{2\ups}(\xbar;\xbar)-\Fups(\xhat;\xbar)] \\
                                        & =8\ups[f(\xbar)-\fups(\xbar)].
    \end{align*}
    \qedsymbol
\end{proof}

The following lemma provides a lower bound on the optimal objective value $f^*$ and a computable upper bound on the QG modulus $\mu$.

\begin{lemma}
    \label{lem:initial_lower-bound}
    Let $f$ be a convex objective function. The following statements are valid.
    \begin{enumerate}
        \item[a)] If $f$ satisfies the QG condition with modulus $\mu>0$, then for any point $x^0 \in X$, we have the lower bound:
              \[
                  f(x^0) - \frac{2}{\mu}\|f'(x^0)\|^2 \le f^*.
              \]
        \item[b)] For any two points $x,y \in X$ with $f(x)>f(y)$, an upper bound on a valid QG modulus $\mu$ for the point $x$ is given by:
              \[
                  \mu \le \frac{2\|f'(x)\|^2}{f(x)-f(y)}.
              \]
        \item[c)] If $f$ is $\bar{\mu}$-strongly convex, then for any point $x^0 \in \mathcal{X}$, we have the tighter lower bound:
              \[
                  f(x^0) - \frac{1}{2\bar{\mu}}\|f'(x^0)\|^2 \le f^*.
              \]
    \end{enumerate}
\end{lemma}
\begin{proof}
    \textit{Part a)}: Let $x_p^*$ be the projection of $x^0$ onto the optimal set $X^*$. By convexity and the QG property, we have:
    \[
        \frac{\mu}{2}\|x^0-x_p^*\|^2 \le f(x^0)-f^* = f(x^0)-f(x_p^*) \le \langle f'(x^0), x^0-x_p^* \rangle \le \|f'(x^0)\|\|x^0-x_p^*\|.
    \]
    Rearranging the first and last terms gives $\|x^0-x_p^*\| \le \frac{2\|f'(x^0)\|}{\mu}$. Substituting this back into the inequality $f(x^0)-f^* \le \|f'(x^0)\|\|x^0-x_p^*\|$ yields the result.

    \textit{Part b)}: The result follows directly from the Polyak-Lojasiewicz (PL) type inequality derived in the Part a), $f(x)-f^* \le \frac{2\|f'(x)\|^2}{\mu}$, combined with the fact that $f^* \le f(y)$.

    \textit{Part c)}: The definition of $\bar{\mu}$-strong convexity provides a quadratic lower bound on $f$ around $x^0$. For any optimal solution $x^*$:
    \[
        f^* = f(x^*) \ge f(x^0) + \langle f'(x^0), x^*-x^0 \rangle + \frac{\bar{\mu}}{2}\|x^*-x^0\|^2.
    \]
    The right-hand side is a quadratic in $(x^*-x^0)$, which is minimized at $x^*-x^0 = -f'(x^0)/\bar{\mu}$. The minimum value is $f(x^0) - \frac{1}{2\bar{\mu}}\|f'(x^0)\|^2$, which gives the desired lower bound on $f^*$.
\end{proof}

\begin{lemma}
    \label{lem:W-certificate-to-G-stationarity}
    For the unconstrained setting ($X=\mathbb{R}^n$), suppose the linear support function $\tilde{l}_f(x;x^i)$ is computed using a gradient at a nearby point $\tilde{x}^i$, i.e., $\nabla\tilde{l}_f(x;x^i)=\nabla f(\tilde{x}^i)$ with $\|x^i-\tilde{x}^i\| \le \bar{\delta}$. Then a $(\iota,\nu)$-\tcW-stationarity certificate for a point $\bar{x}$ is also a $(\iota+\bar{\delta},\nu)$-Goldstein stationarity certificate \cite{zhang2020complexity}.
\end{lemma}
\begin{proof}
    For simplicity, let $\tilde{l}_f(x;x^0) := f(\bar{x}) + \langle f'(\bar{x}), x-\bar{x} \rangle$ denote the linear approximation constructed from the evaluation center $\bar{x}$. Let $\bar{\lambda}$ be the optimal dual variables from the definition of the \tcW-gap, i.e., $\bar{\lambda} := \arg\max_{\lambda\in\Delta_{m+1}^{+}}\min_{x\in B(\bar{x},\iota)}\sum_{i=0}^{m}\lambda_i \tilde{l}_f(x,x^i)$. Let $g = \sum_{i=0}^{m}\bar{\lambda}_i \nabla f(\tilde{x}^i)$. It suffices to show that $\|g\| \le \nu$.

    Since $\mathcal{V}_\iota(\bar{x}) \le \nu$, from the definition of the \tcW-gap we have:
    \begin{align*}
        \psi(\bar{x}) - \sum_{i=0}^{m}\bar{\lambda}_i \tilde{l}_f(x;x^i) & \le \iota\nu, \quad \forall x\in B(\bar{x};\iota).
    \end{align*}
    Using the definition $\psi(\bar{x}) = \max_{i\in[m]} \tilde{l}_f(\bar{x}; x^i)$ and the fact that $\sum \bar{\lambda}_i=1$, we can deduce:
    \begin{align*}
        \sum_{i=0}^{m}\bar{\lambda}_i \left( \tilde{l}_f(\bar{x};x^i) - \tilde{l}_f(x;x^i) \right) & \le \iota\nu                                        \\
        \implies \sum_{i=0}^{m}\bar{\lambda}_i \langle \nabla f(\tilde{x}^i), \bar{x}-x \rangle    & \le \iota\nu                                        \\
        \implies \langle g, \bar{x}-x \rangle                                                      & \le \iota\nu, \quad \forall x \in B(\bar{x};\iota).
    \end{align*}
    By choosing the specific point $x = \bar{x} - \iota \frac{g}{\|g\|}$ from the ball, we get $\iota\|g\| \le \iota\nu$, which simplifies to $\|g\| \le \nu$.\qedsymbol
\end{proof}

\end{document}